\documentclass[11pt]{article}
\usepackage{hyperref}
\usepackage{color}
 \usepackage{mathptmx}      
\usepackage[titletoc,title]{appendix}
\usepackage[numbers]{natbib}
\usepackage[applemac]{inputenc}
\usepackage[all]{xy}
\usepackage{enumerate}
\usepackage{comment,cancel}
\usepackage{amsmath,amssymb}
\usepackage[british]{babel}
\usepackage{a4wide}
\usepackage[sans]{dsfont}
\usepackage{amsfonts}
\usepackage{amsthm}
\usepackage[mathscr]{euscript}
\usepackage{mathrsfs}
\usepackage{latexsym}

\newtheorem{theorem}{Theorem}[section]
\newtheorem{lemma}[theorem]{Lemma}
\newtheorem{proposition}[theorem]{Proposition}
\newtheorem{corollary}[theorem]{Corollary}
\theoremstyle{definition}
\newtheorem{definition}[theorem]{Definition}
\newtheorem{example}[theorem]{Example}
\newtheorem{assumption}[theorem]{Assumption}

\newtheorem{remark}[theorem]{Remark}

\newcommand{\p}{^}

\newcommand{\X}[1]{X\p{#1}}
\newcommand{\N}[1]{X\p{#1}}

\newcommand{\pb}[2]{
	\ensuremath{\langle #1,#2 \rangle}}
	\newcommand{\cv}[2]{
	\ensuremath{[#1,#2]}}

\newcommand{\wt}{\widetilde}

	\newcommand{\rmd}{\mathrm{d}}
\newcommand{\rme}{\mathrm{e}}

\newcommand{\E}{\mathrm{E}}

\newcommand{\Lrm}{\mathrm{L}}

\newcommand{\Ws}{\mathrm{W}\p\sig}

\newcommand{\Bscr}{\mathscr{B}}

\newcommand{\Dscr}{\mathscr{D}}
\newcommand{\Escr}{\mathscr{E}}
\newcommand{\Fscr}{\mathscr{F}}

\newcommand{\Hscr}{\mathscr{H}}

 \newcommand{\Jscr}{\mathscr{J}}
\newcommand{\Kscr}{\mathscr{K}}

 \newcommand{\Xscr}{\mathscr{X}}

\newcommand{\et}{\eta}

 \newcommand{\al}{\alpha}
\newcommand{\bt}{\beta}
\newcommand{\gm}{\gamma}
\newcommand{\dt}{\delta}
\newcommand{\ep}{\varepsilon}
\newcommand{\lm}{\lambda}
\newcommand{\Lm}{\Lambda}
\newcommand{\sig}{\sigma}

\newcommand{\Om}{\Omega}
\newcommand{\cadlag}{c\`adl\`ag }
\newcommand{\Span}{\mathrm{Span}}

\newcommand{\Bb}{\mathbb{B}}
\newcommand{\Ebb}{\mathbb{E}}
\newcommand{\EE}{\mathbb{E}}
\newcommand{\Fbb}{\mathbb{F}}
\newcommand{\Nbb}{\mathbb{N}}
\newcommand{\Pbb}{\mathbb{P}}

\newcommand{\Rbb}{\mathbb{R}}
\newcommand{\Zbb}{\mathbb{Z}}

\newcommand{\scr}[1]{\mathscr{#1}}

\newcommand{\bi}{\begin{itemize}}
\newcommand{\ei}{\end{itemize}}

\newcommand{\be}{\begin{enumerate}}
\newcommand{\ee}{\end{enumerate}}
\newcommand{\beq}{\begin{equation}}
\newcommand{\eeq}{\end{equation}}
\newcommand{\beqs}{\begin{equation*}}
\newcommand{\eeqs}{\end{equation*}}
\newcommand{\beqa}{\begin{eqnarray}}
\newcommand{\eeqa}{\end{eqnarray}}
\newcommand{\beqas}{\begin{eqnarray*}}
\newcommand{\eeqas}{\end{eqnarray*}}

 \newcommand{\aPP}[2]{\ensuremath{\langle #1,#2 \rangle}}

\newcommand{\equa}{\begin{eqnarray*}}
\newcommand{\tion}{\end{eqnarray*}}
\newcommand{\equal}{\begin{eqnarray}}
\newcommand{\tionl}{\end{eqnarray}}

\newcommand{\less}{\hspace{-3em}}

\makeatletter
\def\timenow{\@tempcnta\time
\@tempcntb\@tempcnta
\divide\@tempcntb60
\ifnum10>\@tempcntb0\fi\number\@tempcntb
:\multiply\@tempcntb60
\advance\@tempcnta-\@tempcntb
\ifnum10>\@tempcnta0\fi\number\@tempcnta}
\makeatother

\title{Product and Moment Formulas for Iterated Stochastic Integrals (associated with L\'evy Processes)}

\author{Paolo Di Tella$^1$   \hspace{0.7em} Christel Geiss$^2$    \\  
     }
\date{}
\begin{document}

\maketitle
\begin{abstract} {In this paper, we obtain explicit product and moment formulas for products of iterated integrals generated by families of square integrable martingales associated with 
an arbitrary L\'evy process. We propose a new approach applying  the theory of compensated-covariation stable families of martingales.  Our  main tool is a representation formula for products of elements of a  compensated-covariation stable family, which enables 
to consider L\'evy processes, with both jumps and Gaussian part. }

 \end{abstract}

\vspace{1em} 

{\noindent
\footnotetext[1]{ Institute of Mathematical Stochastics, University of Technology Dresden, Germany. \\ \hspace*{1.5em}
 {\tt Paolo.Di\_Tella{\rm@}tu-dresden.de}}
\footnotetext[2]{University of Jyvaskyla, Department of Mathematics and Statistics, P.O.Box 35, FI-40014 University of Jyvaskyla. \\ \hspace*{1.5em} {\tt christel.geiss{\rm@}jyu.fi}}}


\section{Introduction} 

If $X$  and $Y$ are square integrable martingales, the martingale
$$C(X,Y)_{\,t} := [X,Y]_t-\aPP{X}{Y}_t, \quad t\ge 0,$$
is called  the \emph{compensated-covariation process} of $X$ and $Y$. Here $[X,Y]$ and $\aPP{X}{Y}$  denote the quadratic covariation and the 
predictable quadratic covariation of $X$ and $Y,$ respectively.
A family $\Xscr$ of square integrable martingales is called \emph{compensated-covariation stable} if 
 $C(X,Y) \in\Xscr$ for all $X,Y\in\Xscr$.  
  
Compensated-covariation stability was introduced by Di Tella and Engelbert in \cite{DTE15} to investigate the predictable representation property (PRP) of families of 
martingales.  Di Tella and Engelbert further exploited  this  property in \cite{DTE16} to construct families of martingales possessing the chaotic 
representation property (CRP).

Let $\Xscr:=\{\X{\al}, \, \al\in\Lm \}$ be a family  of  square integrable,  quasi-left continuous martingales,  where $\Lm$ is an arbitrary index set. 
 In \cite{DTE15} the following recursive  representation formula  for products  of elements from  $\Xscr$ was shown, provided that 
 $\Xscr$ is compensated-covariation stable:
  \begin{equation}  \label{product}
\begin{split}
 \prod_{i=1}\p N X\p{\al_{i}}
=&
\sum_{i=1}\p N\sum_{{1\leq j_1<\ldots<j_i\leq N}}\Bigl(  \prod_{\begin{subarray}{c}k=1\\k\neq j_1,\ldots,j_i\end{subarray}}\p N X\p{\al_k}_-\Bigr)\cdot X\p{\al_{j_1},\ldots, \al_{j_i}}
\\
&+
\sum_{i=2}\p N\sum_{1\leq j_1<\ldots<j_i\leq N}\Bigl(\prod_{\begin{subarray}{c}k=1\\k\neq j_1,\ldots,j_i\end{subarray}}\p N X\p{\al_k}_-\Bigr)
\cdot
\pb{X\p{\al_{j_1},...,\al_{j_{i-1}}}}{X\p{\al_{j_{i}}}},
\end{split}
\end{equation}
where   $X^{\al_{j_1}, \al_{j_2}}:= C(X^{\al_{j_1}}, X^{\al_{j_2}})$ and $X\p{\al_{j_1},\ldots, \al_{j_i}}:= C(X\p{\al_{j_1},\ldots,\al_{j_{i-1}}}, X^{\al_{j_i}})$.

Relation \eqref{product} was the corner stone to obtain sufficient conditions on $\Xscr$ for the PRP and the CRP. It will also be the starting point in the present paper to derive a representation formula for products of iterated integrals. 

In the first part of this work, we show that the compensated-covariation stability of the family  $\Xscr$, which is required for \eqref{product}, transfers to the family $\Jscr_\rme$ of the elementary iterated integrals generated by $\Xscr$ (see Definition  \ref{def:sp.mul.int} for $\Jscr_\rme$).  This straightforwardly leads to a a version of the {\it recursive} representation formula  \eqref{product} for products of elements from $\Jscr_\rme$ (see Theorem \ref{thm:prod.for.it.int} below). This  recursive representation formula  turns out to be crucial for deriving our product and moment formulas. Let $X$ be a L\'evy processes such that $\Ebb[|X_t|\p N]<+\infty$, $N\in\Nbb$.  As a by-product of Theorem \ref{thm:prod.for.it.int}, we get in Theorem \ref{thm:mom.lev} below a recursive representation formula for  $\Ebb[X_t\p N]$ in terms of lower moments of $X_t$. This formula seems also to be new and of independent interest.

In the second part of the paper (which is  Section 4),  using the general theory for compensated-covariation stable families developed in the first part as a tool, we obtain product and moment formulas for the iterated integrals generated by  families of martingales associated with a L\'evy process. The main idea is to exploit the recursive formula obtained in Theorem \ref{thm:prod.for.it.int} until we explicitly solve the recursion. 

Let $X$ be a L\'evy process with characteristic triplet $(\gm,\sig\p2,\nu)$ and let $\mu$ be the measure defined by $\mu:=\sig\p2\delta_0+\nu$, where $\delta_0$ is
the Dirac measure concentrated in zero and $\nu$ is the L\'evy measure of $X$. With any system of functions $\Lm\subseteq L\p2(\mu)$ we associate a family 
 $\Xscr_\Lm=\{\X{\al},\ \al\in\Lm\} $ of square integrable martingales setting 
\equa
   X\p\al_t:=\al(0) \sigma W_t +\int_{[0,t]\times \Rbb } \al(x)\tilde N(\rmd s,\rmd x),
\tion
where $W$ is the Brownian motion and $\tilde N$ the compensated Poisson random measure appearing in the L\'evy-It\^o decomposition of $X.$
We recall that, if $\Lm$ is a total system (i.e., the linear hull is dense), then the family $\Xscr_\Lm$ possesses 
the CRP (see \cite[Theorem 6.6]{DTE16}). 
As a first step, we show the product formula for the elementary iterated integrals $\Jscr_\rme$ generated by $\Xscr_\Lm$, provided 
that $\Jscr_\rme$ 
is compensated-covariation stable (cf.\ Theorem \ref{thm:prod.rule} below).  From this product formula 
the moment formula \eqref{moment-formula} is obtained. In Example \ref{calculation} we illustrate formula \eqref{moment-formula} in some special cases.

In Theorem \ref{thm:gen.it.in} below, which is the main result 
of this paper, we extend the product and the moment formula to (non-necessarily elementary) iterated integrals generated by any  family $\Xscr_\Lm$  satisfying $\Lm \subseteq  \bigcap_{p\geq2}L\p p(\mu).$ Especially,  $\Xscr_\Lm$  does  not need to 
be compensated-covariation stable. This extension
is important because, to ensure the compensated-covariation 
stability of $\Xscr_\Lm$, it is necessary to require that 
$\Lm\subseteq L\p2(\mu)$ 
is stable under multiplication. However, in some important situations,
as in the case of an orthonormal basis of $L\p2(\mu)$, the system $\Lm$ may fail to be stable under multiplication.

Thanks to Theorem \ref{thm:gen.it.in}, for any L\'evy process $X$, we  have at our disposal a rich variety of families of martingales $\Xscr_\Lm$ that, according to \cite{DTE16}, possess the CRP,  
and the product and moment formulas hold for the iterated integrals generated by $\Xscr_\Lm$.   \smallskip

We now give an overview of results about product and moment formulas which are available in the literature. 

Russo and Vallois  generalize in \cite{RussoVallois} the well-known product formula of two  iterated integrals generated by a Brownian motion (see \cite{N06}) to a version where 
the  iterated integrals are generated by a normal martingale. To define the iterated integrals Russo and Vallois follow the approach of Meyer \cite{Me76}.

In \cite{LeeShih04}, Lee and Shih introduced the \emph{multiple integrals} associated with a L\'evy process following It\^o \cite{I56} and then they obtained the product formula of two multiple integrals.

We recall that, for L\'evy processes, there is a direct relation between iterated integrals, as defined in \S\ref{subs:el.it.int} below, and the multiple integrals introduced by It\^o in \cite{I56} (see  \cite[Proposition 6.9]{DTE16} and  \cite[Proposition 7]{SU07}).

The relation between moments, cumulants, iterated integrals and orthogonal polynomials was  studied by Sol\'e and Utzet in \cite{SU08}  and \cite{SU081} for the case of Teugels martingales and for L\'evy processes possessing moments of arbitrary order or  an exponential moment.  

 Peccati and Taqqu gave in \cite{PT11} product formulas for iterated integrals generated by a Brownian motion. In this case, the basic tool is the \emph{hypercontractivity} of the iterated integrals. However, hypercontractivity is typical for the Brownian case and it does not hold, for instance, for iterated integrals generated by a compensated Poisson process. This is the reason why for L\'evy processes, in general, additional integrability  conditions on the integrands are required  if product or moment formulas are considered. 

In Peccati and Taqqu \cite{PT11} also moments and cumulants of products for multiple integrals generated by a Poisson random measure have been investigated. This was done on the basis of the Mecke formula (see Mecke \cite{M67}) and diagram formulas. For similar results we refer  also to  Surgailis \cite{S84}. In Last et al.\ \cite{LPST14}, again on the basis of the Mecke formula and of diagram formulas, more general product and moment formulas for multiple iterated integrals than those in \cite{PT11} and \cite{S84} have been obtained. We stress that all these results cannot be applied for  L\'evy processes with a Gaussian part, because in this case the Mecke formula is not applicable.

First results about the product and the moment formulas for L\'evy processes with non-vanishing Gaussian part, were obtained by Geiss and Labart in \cite{GL16,GL17}, where  iterated integrals generated by a simple L\'evy process (i.e., the sum of a Brownian motion and of a compensated Poisson process) are considered. 
The moment formula  in Theorem \ref{thm:gen.it.in} below generalizes \cite{GL16,GL17} to arbitrary  L\'evy processes. 

This paper has the following structure: We summarise the necessary background in Section \ref{sec:prel}.  In Section \ref{sec:mom.el.it} we study the properties of the elementary iterated integrals generated by a compensated-covariation stable family of martingales. In Section \ref{sec:Lev.pr} we obtain product and moment formulas for the iterated integrals generated by families of martingales associated with a L\'evy process. We then conclude giving concrete examples of these families, including Teugels martingales.

\section{Preliminary Notions}\label{sec:prel}
Let $(\Om,\Fscr,\Pbb)$ be a \emph{complete} probability space and let $\Fbb$ be a filtration satisfying the \emph{usual conditions}. We shall always consider real-valued stochastic processes on a finite time horizon $[0,T]$, $T>0,$ and  assume that $\Fscr= \Fscr_T.$

Let $\Xscr$ be a family of processes. By $\Fbb\p\Xscr$ we denote the smallest filtration satisfying the usual conditions such that $\Xscr$ is an adapted family. If $\Xscr=\{X\}$ we write $\Fbb\p\Xscr=\Fbb\p X$.

We say that a process $X$ has a finite moment of order $N$ if $\Ebb[|X_t|\p N]<+\infty$, $t\in[0,T]$. If this holds for every $N\in\Nbb$, we say that $X$ has finite moments of every order.

For a \cadlag process $X$, we denote by $\Delta X:=X-X_-$ the \emph{jump process} of $X$, with $X_{t-}:=\lim_{s\uparrow t}X_s$ for $t>0$, and, by convention, $X_{0-}:=X_0$ so that $\Delta X_0=0$.

In the present paper, $\Fbb$-martingales are always assumed to be c\`adl\`ag and starting at zero. 

A martingale $X$ belongs to $\Hscr\p p$ if $\|X\|_p\p p:=\Ebb[|X_T|\p p]]<+\infty$, $p\geq1$, and  $(\Hscr\p p,\|\cdot\|_{\Hscr\p p})$ is a Banach space for every $p\geq1$ and a Hilbert space for $p=2$. We sometimes write $\Hscr\p p(\Fbb)$ to specify the filtration. 

If $X,Y\in\Hscr\p{\,2}$, then there exists a unique predictable process of integrable variation, denoted by $\pb{X}{Y}$ and called the \emph{predictable covariation} of $X$ and $Y$, such that $XY-\pb{X}{Y}\in\Hscr\p1$ and $\pb{X}{Y}_0=0$. For $X\in\Hscr\p2$, the process $\aPP{X}{X}$ has a continuous version if and only if $X$ is a quasi-left continuous martingale (cf.\ \cite[Theorem I.4.2]{JS00}), that is, $\Delta X_\tau=0$ for every \emph{predictable} stopping time $\tau$.

With two semimartingales $X$ and $Y$, we associate the process $[X,Y]$, called covariation of $X$ and $Y$, defining
\begin{equation*}
[X,Y]_{\,t}:=\pb{X\p c}{Y\p c}_{\,t}+\sum_{0\leq s\leq t}\Delta X_{\,s}\Delta Y_{\,s},\quad t\in[0,T],
\end{equation*}
where $X\p c$ and $Y\p c$ denote the continuous martingale part of $X$ and $Y$ respectively. 
For $X,Y\in\Hscr\p{\,2}$, the process $[X,Y]$ is of integrable variation and $[X,Y]-\pb{X}{Y}\in\Hscr\p1$, that is $\pb{X}{Y}$ is the \emph{predictable compensator} of $[X,Y]$ (\cite[Proposition I.4.50 b)]{JS00}).

We recall the definition of the stochastic integral with respect to a martingale $X\in\Hscr\p{\,2}$. The space of integrands for $X$ is given by
$\Lrm\p{2}(X):=\{H\textnormal{ predictable}: \Ebb[H\p{\,2}\cdot\pb XX_T]<+\infty\}$,
where we denote $H\p{\,2}\cdot\pb XX_T:=\int_0\p TH\p2_s\rmd \aPP{X}{X}_s$. For $X\in\Hscr\p{\,2}$ and $H\in\Lrm\p{2}(X)$
 we denote  by $H\cdot X$ or $\int_0\p\cdot H_s\rmd X_s$ the stochastic integral of $H$ with respect to $X$, characterized as follows: Let $Z\in\Hscr\p2$. Then $Z=H\cdot X$ if and only if 
$\pb ZY=H\cdot\pb XY$, for every $Y\in\Hscr\p{\,2}$. 

Let $\Lm$ be an arbitrary parameter set and $n\leq m$ natural numbers. By $\al_{n:m}$ we denote the ordered $(m-n+1)$-dimensional tuple  $(\al_n,\ldots,\al_m)\in\Lm\p {m-n+1}$.
For  $\al_{n_1:m_1}\in \Lm\p {m_1-n_1+1}$ and $\bt_{n_2:m_2}\in\Lm\p {m_2-n_2+1}$, we denote by $\al_{n_1:m_1},\bt_{n_2:m_2}$ the $(m_1-n_1+1)+(m_2-n_2+1)$-dimensional 
tuple obtained by continuing with $\bt_{n_2:m_2}$ after $\al_{n_1:m_1}$, that is, $\al_{n_1:m_1},\bt_{n_2:m_2}:=(\al_{n_1},\ldots,\al_{m_1},\bt_{n_2},\ldots,\bt_{m_2})$.

We use the following notation: For any measure $\rho$ and any function $f\in L\p1(\rho)$, we denote by $\rho(f)$ the integral of $f$ with respect to $\rho$, that is,
$
\rho(f):=\int f(x)\,\rho(\rmd x)$.

\section{Iterated integrals and compensated-covariation stability}\label{sec:mom.el.it}
In this section we introduce compensated-covariation stable families of martingales and iterated integrals generated by such families. We show that if $\Xscr\subseteq\Hscr\p2$ is a compensated-covariation stable family, then the family of elementary iterated integrals generated by $\Xscr$ is compensated-covariation stable as well. Using this property, we deduce a formula to represent products and moments of iterated integrals.

 \subsection{Iterated integrals}\label{subs:el.it.int}
To begin with, we define elementary iterated integrals generated by a finite family of martingales. For  $\X\al\in\Hscr\p2$ with deterministic point brackets  $\aPP{\X{\al}}{\X{\al}}$, we denote 
\equal \label{rho}
\rho\p{\al} (\rmd t) := \rmd\aPP{\X{\al}}{\X{\al}}_t, \quad t\in[0,T]. 
\tionl

Let now $m\in\Nbb$ be given and  $\X{\al_1},\ldots,\X{\al_m}\in\Hscr\p2$ be such that the predictable covariation  $\aPP{\X{\al_j}}{\X{\al_k}}$ is a deterministic function for $j,k=1,\ldots,m$. Then we   denote  by $\rho\p{\al_{1:m}}$  the product measure 
\equal \label{product-rho}
   \rho\p{\al_{1:m}} := \rho\p{\al_1}\otimes \ldots \otimes \rho\p{\al_m}
\tionl
 on  $([0,T]\p m,\Bscr([0,T])\p m)$. 

\begin{definition}\label{elem.it.int.}
\textnormal{(i)} The space of bounded measurable functions on $[0,T]$ will be denoted by $\Bb_T$. 
 For $m\ge 1$ we introduce the tensor product 
$$ \Bb_T\p {\otimes m}:=\{F_{\otimes_m}=F_1\otimes\cdots\otimes F_m:  F_1,\ldots,F_m \in \Bb_T  \} $$  
and call the elements $F_{\otimes_m}$ \emph{elementary functions of order} $m$. For $m=0$ put $F_{\otimes_0}:=1$. 

\textnormal{(ii)}  Let $m\geq0$ and the martingales $\X{\al_1}, \X{\al_2}, \ldots, \X{\al_m}\in\Hscr\p2$ be fixed. For all $0\leq n\leq m$, the $n$\emph{-fold elementary iterated integral} of $F_{\otimes_n} \in\Bb_T\p {\otimes n} $ 
with respect to the martingales $(\X{\al_1}, \X{\al_2}, \ldots, \X{\al_n})$ is 
defined inductively by letting  $J_{0}:\Rbb \to \Rbb $ be the identical map and
\begin{equation}\label{eq:def.sim.mul.int}
J_{n}\p{\al_{1:n}}(F_{\otimes_n})_{\,t}
:=\int_0^t J_{n-1}^{\al_{1:n-1}}(F_{\otimes_{n-1}})_{u-}\,F_{n}(u)\, \rmd {\X{\al_{n}}_u},\quad t\in [0,T]\,.
\end{equation}
\end{definition}
\bigskip

The following properties of elementary iterated integrals are shown in   \cite[Lemma 3.2]{DTE16}.
\begin{proposition}\label{prop:isom.sim.mul.int} 
\textnormal{(i)} Let $m\geq 1$, $\X{\al_1}, \ldots\ , \X{\al_m} \in \Hscr\p2$  such that $\aPP{\X{\al_j}}{\X{\al_k}}$ is a deterministic function, $j,k=1,\ldots,m$. Then we have $J_{m}\p{\al_{1:m}}(F_{\otimes_m}) \in \Hscr\p2$ for $m\geq1$.

\textnormal{(ii)} Let $n\geq 1$, $\X{\bt_1}, \ldots\ , \X{\bt_n} \in \Hscr\p2$  such that $\aPP{\X{\bt_j}}{\X{\bt_k}}$ is a deterministic function, $j,k=1,\ldots,n$. 
Then for any    $F_{\otimes_m}  \in  \Bb_T\p {\otimes m} $ and    $G_{\otimes_n} \in \Bb_T\p {\otimes n}$ it holds $\Ebb\big[J_{m}^{\al_{1:m}}(F_{\otimes_m})_{\,t}J_{n}^{\bt_{1:n}}(G_{\otimes_n})_{\,t}\big]=0$, if $ m\neq n$, while, if $m=n$,
\beq\label{eq:IRmix.el}
\begin{split}
\Ebb\big[J_{m}^{\al_{1:m}}(F_{\otimes_m})_{\,t}&J_{m}^{\bt_{1:m}}(G_{\otimes_m})_{\,t}\big]\\
=&\displaystyle\int_0\p t\int_0\p {t_m-}\!\!\!\!\!\!\!\!\!\cdots\int_0\p{t_2-}\!\!\!\!F_1(t_1)G_1(t_1)\ldots F_ m(t_m)G_ m(t_m)\,\rmd\aPP{\X{\al_1}}{\X{\bt_1}}_{t_1}\ldots
\rmd\aPP{\X{\al_m}}{\X{\bt_m}}_{t_m}\,.
\end{split}
\end{equation}
\end{proposition}
The next lemma concerns some properties of elementary iterated integrals, which will be useful in this paper. 
\begin{lemma}\label{lem:mom.it.int} Let  $\X{\al_j}\in \Hscr\p2$ have moments of every order for $j=1,\ldots,m$.

\textnormal{(i)}
  For every $F_{\otimes_m} \in \Bb_T\p {\otimes m}$ and every  $p>0$ we have the estimate
\equal \label{sup}
\Ebb\Bigg[\sup_{t\in[0,T]} \Big|J\p{\al_{1:m}}_{m}(F_{\otimes_m})_t\Big|\p{p}\Bigg]<\infty.
\tionl

\textnormal{(ii)}  Assume that  $p \geq1.$  Let $\X{\al\p n_1},\ldots,\X{\al\p n_m}  \in \Hscr\p2$ be such that 
\smallskip

\hspace{.5cm}\textnormal{(1)} $\aPP{\X{\al\p n_j}}{\X{\al\p n_k}}$ is a deterministic function for 
every $j,k=1,\ldots,m$ and  $n\ge1$;

 \hspace{.5cm}\textnormal{(2)} each  $\X{\al\p n_j}$ has finite moments of every order;

\hspace{.5cm}\textnormal{(3)}  $\X{\al\p n_j}\longrightarrow\X{\al_j}$ in $\Hscr\p{2\p m p}$, 
as $n\rightarrow+\infty$. 
\smallskip

Then,  for every $F_{\otimes_m} \in \Bb_T\p {\otimes m}$, we have
\[
\lim_{n\rightarrow+\infty}\Ebb\left[\sup_{t\in[0,T]}\Big|J_m\p{\al_{1:m}\p n}(F_{\otimes_m})_t-J_m\p{\al_{1:m}}(F_{\otimes_m})_t\Big|\p{p}\right]=0.
\]
\end{lemma} 
\begin{proof}
We prove both the statements by induction on $m$. We start proving (i). We only consider the case $p\geq1$, since the case $0<p<1$ immediately follows by the case $p=1$ and Jensen's inequality. If $m=0$ there is nothing to prove. We now assume  that \eqref{sup} holds  for an arbitrary  $p \geq1$ and every $j\leq m,$ and we show it for $m+1$.  By Burkh\"older--Davis--Gundy's inequality, from now on BDG's  inequality, (see \cite[Theorem 2.34]{J79})
\equa
 \Ebb\Bigg[\sup_{t\in[0,T]}\Big|J\p{\al_{1:{m+1}}}_{{m+1}}(F_{\otimes_{m+1}})_t\Big|\p{p}\Bigg] 
&\leq& 
C_{p} \Ebb\Bigg[\Big[J\p{\al_{1:{m+1}}}_{{m+1}}(F_{\otimes_{m+1}}),J\p{\al_{1:{m+1}}}_{{m+1}}(F_{\otimes_{m+1}})\Big]_T\p{p/2}\Bigg] \\
&=&
C_{p}\Ebb\Bigg[\Big(J\p{\al_{1:{m}}}_{{m}}(F_{\otimes_{m}})_-\p2F_{m+1}\p2\Big)\cdot\big[\X{\al_{m+1}},\X{\al_{m+1}}\big]_T\Big)\p {p/2}\Bigg] \\
&\leq&
C_{p} c_{m+1}\p p \Ebb\Bigg[\sup_{t\in[0,T]} \Big|J\p{\al_{1:{m}}}_{{m}}(F_{\otimes_{m}})_t\Big|\p{2p}\Bigg]\p{1/2}\!\!\!
\Ebb\Big[\big[\X{\al_{m+1}},\X{\al_{m+1}}\big]_T\p{p}\Big]\p{1/2}\,,
\tion
 where ${c_{m+1}=\sup_{t \in[0,T]} |F_{m+1}(t)|<+\infty}$. Notice that 
\[\Ebb[[\X{\al_{m+1}},\X{\al_{m+1}}]_T\p{p}\Big] \le 
c_{(2p)} \Ebb[|\X{\al_{m+1}}_T|^{2p}] < \infty,
\]
where the constant $c_{(2p)}$ arises from the use  of  BDG's inequality  and Doob's martingale inequality.
Hence, the right hand side of the above estimate is finite by the induction hypothesis.  The proof of (i) is complete. Concerning (ii) we observe that
\equa
&&\Ebb\Bigg[\sup_{t\in[0,T]}\Big|J\p{\al\p n_{1:{m}}}_{{m}}(F_{\otimes_{m}})_t-J\p{\al_{1:{m}}}_{{m}}(F_{\otimes_{m}})_t\Big|\p{p}\Bigg] \\
&& \quad\leq
2\p {p-1}\Bigg\{\Ebb\Bigg[\sup_{t\in[0,T]}\bigg|\int_0\p t\Big(J\p{\al\p n_{1:{m-1}}}_{{m-1}}
(F_{\otimes_{m-1}})_{u-}-J\p{\al_{1:{m-1}}}_{{m-1}}(F_{\otimes_{m-1}})_{u-}\Big)F_m(u)\rmd \X{\al\p n_m}_u\bigg|\p p\Bigg]\\
&&\hspace{3cm}+
\Ebb\Bigg[\sup_{t\in[0,T]}\bigg|\int_0\p tJ\p{\al\p n_{1:{m-1}}}_{{m-1}}(F_{\otimes_{m-1}})_{u-}F_m(u)\rmd \big(\X{\al\p n_m}_u-\X{\al_m}_u\big)\bigg|\p p\Bigg]\Bigg\} \\
&&\quad\leq2\p {p-1}C_pc_m\p p\Bigg\{\Ebb\bigg[\sup_{t\in[0,T]}\Big|\big(J\p{\al\p n_{1:{m-1}}}_{{m-1}}(F_{\otimes_{m-1}})_{t}-J\p{\al_{1:{m-1}}}_{{m-1}}(F_{\otimes_{m-1}})_{t}\big)\Big|\p {2p}\bigg]\p{1/2}\Ebb\Big[\big[\X{\al\p n_m},\X{\al\p n_m}\big]_T\p p\Big]\p{1/2}
\\&&\hspace{3cm}+
\Ebb\bigg[\sup_{t\in[0,T]}\Big|J\p{\al\p n_{1:{m-1}}}_{{m-1}}(F_{\otimes_{m-1}})_{t}\Big|\p {2p}\bigg]\p{1/2}\Ebb\Big[ \big[\X{\al\p n_m}-\X{\al_m},\X{\al\p n_m}-\X{\al_m}\big]_T\p p \Big]\p{1/2}\Bigg\},
\tion
where we similarly as above used  BDG's inequality and H\"older's inequality. Because of the convergence assumptions, we see that $\Ebb\big[\big[\X{\al\p n_m},\X{\al\p n_m}\big]_T\p p\big]$ is bounded in $n$ and $\Ebb\big[\big[\X{\al\p n_m}-\X{\al_m},\X{\al\p n_m}-\X{\al_m}\big]_T\p p\big]$ converges to zero, as $n\rightarrow+\infty$. Furthermore, because for $m=1$ we have
\[
\Ebb\Bigg[\sup_{t\in[0,T]}\Big|J_1\p{\al\p n_{{1}}}(F_1)_t-J_1\p{\al_{1}}(F_{1})_t\Big|\p{2\p mp}\Bigg]\longrightarrow 0,\quad n\rightarrow+\infty,
\] 
we can assume, for every $k\leq m-1$,
\[
\Ebb\Bigg[\sup_{t\in[0,T]}\Big|J\p{\al\p n_{1:{k}}}_{{k}}(F_{\otimes_{k}})_t-J\p{\al_{1:{k}}}_{{k}}(F_{\otimes_{k}})_t\Big|\p{2\p{m-k}p}\Bigg]\longrightarrow 0,\quad n\rightarrow+\infty.
\]
Hence, (ii) follows by induction and the proof of the lemma is complete.
\end{proof}

As a next step we introduce some linear spaces of elementary iterated integrals and the definition of the chaotic representation property.
\begin{definition}\label{def:sp.mul.int}
 Let $\Xscr:=\{\X{\al},\, \, \al\in\Lm\}\subseteq\Hscr\p2$, where  $\Lm$  is an arbitrary parameter set. We assume that $\aPP{\X{\al}}{\X{\bt}}$ is a deterministic function, for each $\al,\bt\in\Lm$.

(i) Let $\Jscr_{0}:=\Rbb$ be the space of $0$-fold elementary iterated integrals.
  
(ii) Let $m\geq1$ and $\al_1, \ldots, \al_m\in\Lm$.  By $\Jscr_{m}^{\al_{1:m}}$ we denote the linear hull of  all $J_{m}^{\al_{1:m}}(F_{\otimes_m})$ 
 with respect to $(\X{\al_1}, \X{\al_2}, \ldots, \X{\al_m})$ from $\Hscr\p2$ and  $F_{\otimes_m} \in \Bb_T\p {\otimes m}$. 

(iii) For all $m\geq1$, we introduce
\begin{equation*}
\quad\scr J_{m}\!\!:=\!\!\Span\Big(\bigcup_{\al_{1:m}\in \Lm^m} \scr J_{m}^{\al_{1:m}}\Big),\qquad\scr J_{\rme}\!\!:=\!\!\Span\Big(\bigcup_{m\geq 0} \scr J_{m}\Big)\,,
\end{equation*}
We call $\scr J_{\rme}$ the space of elementary iterated integrals \emph{generated by} $\Xscr$.

(iv) If $\scr J_{\rme}$ is dense in $(\Hscr\p2(\Fbb),\|\cdot\|_2)$, we say that $\Xscr$ possesses the \emph{chaotic representation property} (CRP) with respect to  $\Fbb$. 
\end{definition}

\begin{remark}\label{rem:alt.rep.crp} We now briefly recall an alternative formulation of the CRP. For details we refer to \cite[Proposition 3.7 and Theorem 3.11]{DTE16}. 

(i) From Proposition \ref{prop:isom.sim.mul.int} (ii), we deduce the equivalent representation $\Jscr_\rme=\bigoplus_{m\geq0}\Jscr_m$. 
Denoting now by $\Jscr$ and $\tilde\Jscr_m$ respectively the closure of $\Jscr_\rme$ and $\Jscr_m$ in $(\Hscr\p2,\|\cdot\|_2)$, we 
see that $\Jscr=\bigoplus_{m\geq0}\tilde\Jscr_m,$ and that $\Xscr$ has the CRP if and only if $\Hscr\p2(\Fbb)=\bigoplus_{m\geq0}\tilde\Jscr_m$. 

(ii) If furthermore $\Xscr$ consists of countably many  martingales (that is, $\Lm$ is a countable index set) which are pairwise orthogonal,  we can write
\[
\Jscr=\bigoplus_{m\geq0}\bigoplus_{\al_{1:m}\in \Lm^m} \tilde\Jscr_{m}^{\al_{1:m}},
\]
$\tilde\Jscr_{m}^{\al_{1:m}}$ denoting the closure of $\Jscr_{m}^{\al_{1:m}}$ in $(\Hscr\p2,\|\cdot\|_2)$.
\end{remark}
\bigskip

We conclude this section introducing \emph{the iterated integrals} as an isometric extension of the elementary iterated integrals. For $t\in[0,T]$ and $m\geq 1$, we introduce the set
\begin{equation}\label{eq:def.Mn}
 M_{\,t}^{\, m}:=\{(t_1,\ldots,t_m): \ 0\leq t_1  <\ldots  < t_m < t\}. 
\end{equation}

 We recall the  definition of $\rho\p{\al}$   and $\rho\p{\al_{1:m}}$ given  in \eqref{rho}  and \eqref{product-rho}, respectively. 
 For $m\geq1$, we denote by $\Escr_t^m$ the linear subspace of $L^2({M}_{\,t}^{\, m}, \rho^{\al_{1:m}})$ generated by 
$F_{\otimes_m} \in \Bb_T\p {\otimes m}$ restricted to ${M}_{\,t}^{\, m}$. Notice that $\Escr_t^m$ is dense in $L^2({M}_{\,t}^{m}, \rho^{\al_{1:m}})$. 
From \eqref{eq:IRmix.el} we have the   isometry relation
\begin{equation}\label{eq:isom}
 \|J_m\p{\al_{1:m}}(F)_t\|_{L\p2(\Pbb)}=
\|F\|_{L^2({M}_{\,t}^{\, m},\, \rho^{\al_{1:m}})}\,
\end{equation} between $\{ F:= F_{\otimes_m} 1_ {{M}_{\,t}^{\, m}}:        F_{\otimes_m}   \in \Bb_T\p {\otimes m}  \}$
and $L\p2(\Pbb)$. We linearly  extend  $J_m\p{\al_{1:m}}(\cdot)_t$  to  $\Escr_t^m$, and by continuity, to $L^2({M}_{\,t}^{m}, \rho^{\al_{1:m}})$,
denoting this extension again by $J_m\p{\al_{1:m}}(\cdot)_t$.  Recall that  $J_0$  denotes the identity map on $\Rbb$.

\begin{definition}
We call the  mapping $J_m\p{\al_{1:m}}(F)$,  $F\in L^2({M}_{\,T}^{\,m}, \rho^{\al_{1:m}})$,  defined above an \emph{iterated integral} with respect to 
$(X\p{\al_1},\ldots, X\p{\al_m})$. 
\end{definition}
In the next proposition we summarize the  properties of iterated integrals.

\begin{proposition}\label{prop:isom.mul.int} 
 Let $m\geq1$, $\al_1,\ldots,\al_m\in\Lm$, $F\in L^2({M}_{\,T}^{\,m}, \rho^{\al_{1:m}})$ and let $\X{\al_1}, \ldots\ , \X{\al_n} \in \Hscr\p2$ be such that $\aPP{\X{\al_j}}{\X{\al_k}}$ is a deterministic function for $j,k=1,\ldots,m$.  \\
 \textnormal{(i)}  $J_{m}\p{\al_{1:m}}(F_{\otimes_m})$ belongs to $ \Hscr\p2$ and it is quasi-left continuous. \\
\textnormal{(ii)} Let moreover $n\geq1$, $\bt_1, \ldots, \bt_n\in \Lm$, $G\in L^2({M}_{\,T}^{\,n}, \rho^{\bt_{1:n}})$ and let $\X{\bt_1}, \ldots\ , \X{\bt_n} \in \Hscr\p2$ be such that $\aPP{\X{\bt_j}}{\X{\bt_k}}$ is a deterministic function, $j,k=1,\ldots,n$. Then, for every $t\in[0,T]$, 
we have: If $ m\neq n$, then $\Ebb\big[J_{m}^{\al_{1:m}}(F)_{\,t}J_{n}^{\bt_{1:n}}(G)_{\,t}\big]=0$, while, if $m=n$,
\equa
\E \big[J_{m}^{\al_{1:m}}(F)_{\,t}\,J_{m}^{\bt_{1:m}}(G)_{\,t}\big]
= \int_0\p t\int\p {t_{m}-}_0\!\!\cdots\int_0\p{t_2-}\!\!
F(t_1,\ldots,t_m)\,G(t_1,\ldots,t_m)\,\rmd\aPP{\X{\al_1}}{\X{\bt_1}}_{t_1}\ldots
\,\rmd\aPP{\X{\al_m}}{\X{\bt_n}}_{t_m}\,.
\tion
\end{proposition}

\subsection{Compensated-covariation stable families of martingales}
\begin{definition} 
Let $\Lm$ be an arbitrary parameter set and $\Xscr:=\{X\p{\al}, \al\in\Lm\}\subseteq\Hscr\p2$.

(i) For every $\al_1,\al_2\in \Lm$ we define the process
\begin{equation}\label{def:compcov}
X\p{\al_{1:2}}:=\cv{X\p{\al_1}}{X\p{\al_2}}-\pb{X\p{\al_1}}{X\p{\al_2}}
\end{equation}
which we call \emph{the compensated-covariation} process of $X\p{\al_1}$ and $X\p{\al_2}$. 

(ii) The family $\Xscr$ is called \emph{compensated-covariation stable} if for every $\al_1, \al_2 \in \Lm$ it holds that  $X\p{\al_{1:2}}  \in \Xscr$.

(iii) Let $\Xscr$ be a compensated-covariation stable family and let $\al_1,\ldots,\al_m\in \Lm$ with $m\geq2$. The process $X\p{\al_{1:m}}$ is defined recursively by
\begin{equation}\label{eq:it.com.cov}
X\p{\al_{1:m}}:=[X\p{\al_{1:m-1}},X\p{\al_{m}}]
-\pb{X\p{\al_{1:m-1}}}{X\p{\al_{m}}}
\end{equation}
and called $m$-fold compensated-covariation of the ordered $m$-tuple of martingales $(X\p{\al_1},\ldots,X\p{\al_m})$. We sometimes use also the notation $X\p{\al_1,\ldots,\al_m}$  instead of 
$X\p{\al_{1:m}}.$
\end{definition}

Notice that if the family $\Xscr$ is compensated-covariation stable, then  $X\p{\al_{1:m}} \in \Xscr$ holds, for every $\al_1,\ldots,\al_m$ in $\Lm$ and $m\geq2$. 

As a toy-example of a compensated-covariation stable family consider the family $\Xscr:=\{X\}$, where $X_t=N_t-\lm t $, $t\geq0$, and $N$ is a  homogeneous Poisson process with parameter $\lm$. In this case we have $\aPP{X}{X}_t=\lm t$ and
\[
[X,X]_t-\aPP{X}{X}_t=\sum_{s\leq t}(\Delta N_s)\p 2-\lm t=\sum_{s\leq t}\Delta N_s-\lm t=X_t.
\]
More generally, let $N$ be a simple point process, i.e., $N$ is a \cadlag adapted increasing process taking value in the set $\Nbb$ of natural numbers such that $N_0=0$ and $\Delta N\in\{0,1\}$. Let  $N\p p$  be the predictable compensator of $N$ and assume that $N\p p$ is continuous (i.e., that $N$ is quasi-left continuous). Then the family$\Xscr:=\{X\}\subseteq\Hscr\p2_\mathrm{loc}$, where $X:=N-N\p p$, is compensated-covariation stable.

  Next we state a representation formula for products of martingales from  $\Hscr\p2$.   Thanks to  the \emph{quasi-left continuity} which we assume here,
  the  representation formula  is  a simpler version of the one shown in
  \cite[Proposition 3.3]{DTE15}.

\begin{proposition}\label{prop:rep.pol}
Let $\Xscr:=\{X\p{\al},\ \al\in\Lm\}$ be a compensated-covariation stable family of quasi-left continuous martingales in $\Hscr\p{\,2}$. For every $N\geq1$ and $\al_1,\ldots,\al_N\in \Lm$, we have
\begin{equation}\label{eq:rep.pol}
\begin{split}
 \prod_{i=1}\p N X\p{\al_{i}}
=&
\sum_{i=1}\p N\sum_{{1\leq j_1<\ldots<j_i\leq N}}\Bigl(  \prod_{\begin{subarray}{c}k=1\\k\neq j_1,\ldots,j_i\end{subarray}}\p N X\p{\al_k}_-\Bigr)\cdot X\p{\al_{j_1:j_i}}
+
\sum_{i=2}\p N\sum_{1\leq j_1<\ldots<j_i\leq N}\Bigl(\prod_{\begin{subarray}{c}k=1\\k\neq j_1,\ldots,j_i\end{subarray}}\p N X\p{\al_k}_-\Bigr)
\cdot
\pb{X\p{\al_{j_1:j_{i-1}}}}{X\p{\al_{j_{i}}}}.
\end{split}
\end{equation}
\end{proposition} 

\begin{remark} For compensated-covariation stable families with deterministic point brackets,  sufficient conditions for the CRP (see Definition \ref{def:sp.mul.int} (iv)) are known. We  briefly recall these results, for an extensive study of which we refer to \cite{DTE16}. 

 The set $\Kscr$ of polynomials generated by $\Xscr$ is defined as the linear hull of products of elements of $\Xscr$ taken at different \emph{deterministic} times. Clearly, the completed $\sig$-algebra generated by $\Kscr$ coincides with $\Fscr\p\Xscr_T$. In \cite[Theorem 5.8]{DTE16},  the following result has been established: \medskip

The family  $\Xscr\subseteq\Hscr\p2(\Fbb\p\Xscr)$ possesses the CRP with respect to $\Fbb\p\Xscr$ provided that

{\rm(i)} $\Xscr$ is compensated-covariation stable;

{\rm(ii)} $\aPP{\X{\al_1}}{\X{\al_2}}$ is deterministic, for every $\al_1,\al_2\in\Lm$;

{\rm(iii)} the family $\Kscr$ of polynomials generated by $\Xscr$ is dense in $L\p2(\Om,\Fscr\p\Xscr_T,\Pbb)$.  \medskip

We recall that to ensure (iii), it is sufficient that for every $X\in\Xscr$, the random variable $|X_t|$ possesses an (arbitrarily small) exponential 
moment for every $t\in[0,T]$: That is, there exist $c\p X_t>0$ such that $\Ebb[\exp(c\p X_t|X_t|)]<+\infty$, for every $t\in[0,T]$. For an elementary proof of this well-known result see,  for example, \cite[Appendix A]{DTE15}.
\end{remark}
\subsection{Representing powers of processes}\label{subs:exa}
  As an application of Proposition \ref{prop:rep.pol} we show how to represent powers of a Brownian motion, of a homogeneous Poisson process and then, more generally, of a L\'evy process with  finite moments of every order.

\paragraph{Brownian Motion.} Let $X$ be a standard Brownian motion. Then $\Xscr=\{X\}\cup\{0\}$ is a compensated-covariation stable family. Furthermore $\aPP{X}{X}_t=t$ and $X\p{\al_{1:m}}=0$, for $m\geq2$. Therefore \eqref{eq:rep.pol} becomes
\begin{equation}\label{eq:prod.for.one.mart.bm}
X\p N_t=N \int_0\p tX_s\p {N-1}\rmd X_s+\frac{N(N-1)}{2}\,\int_0\p t X_s\p {N-2}\rmd s, 
\end{equation}
which is in fact It\^o's formula applied to $X\p N$.

\paragraph{Homogeneous Poisson Process.}  We now denote by $N$ be a homogeneous Poisson process with parameter $\lm>0$. Then, setting $X_t:=N_t-\lm t$, $\Xscr=\{X\}$ is a compensated-covariation stable family. Furthermore, $X\p{\al_{1:m}}=X$, for every $m\geq1$. From \eqref{eq:rep.pol} we deduce
\begin{equation}\label{eq:prod.for.one.mart.poi.pr}
X\p N_t=\sum_{i=1}\p N \binom{N}{i} \int_0\p t X_{s-}\p {N-i}\rmd X_s+\lm\sum_{i=2}\p N \binom{N}{i}\int_0\p t X_{s-}\p {N-i}\rmd s.
\end{equation}

\paragraph{L\'evy Processes.} 
 Let $\Fbb$ be a filtration satisfying the usual conditions and let $X$ be a L\'evy process relative to $\Fbb$,  that is: $X$ is c\`adl\`ag, stochastically continuous, $\Fbb$-adapted,  $X_{t+s}-X_s$   is independent from $\Fscr_s$ and  $X_{t+s}-X_s\sim X_t$, for all $s,t\geq0$, and $X_0=0.$ The L\'evy-It\^o decomposition of $X$ is  given by
 \begin{equation}\label{eq:leit.dec} 
X_t = \gamma t + W_t^\sigma + \int_{[0,t]\times\{|x|>1\} } x\, N(\rmd s,\rmd x) + \int_{[0,t]\times\{|x|\le1\} } x\,\tilde N(\rmd s,\rmd x),  
\end{equation} 
where $\gm\in\Rbb$, $\Ws$ denotes a Wiener process relative to $\Fbb$ with variance function $\Ebb[(\Ws_t)\p2]=\sig\p2 t$, and $N$ is the jump measure of $X$, which is a Poisson random measure relative to $\Fbb$. By $\tilde N:=N-\nu\otimes\lm_+$ we denote the compensated Poisson random measure associated with $N$, where $\lm_+$ denotes the Lebesgue measure on $[0,T]$ and $\nu$ the L\'evy measure of $X$ (i.e., $\nu$ is a $\sig$-finite measure on $\Rbb$ such that $\nu(\{0\})=0$ and $x\mapsto x\p2\wedge 1\in L\p1(\nu)$). We call $(\gm,\sig\p2,\nu)$ the characteristic triplet of $X$. \\ \smallskip

We now assume that $X$ has finite moments of every order. This  implies that, for $p_i(x):=x\p i$, we have   $p_i\in L\p2(\nu)$ for all $i\geq1$ and $p_i\in L\p1(\nu)$ for all $i\geq2$.

 As in \cite{NS00} or \cite{SU08} we now define the power-jump processes of $X$ and the family of Teugels martingales, setting $L\p{(1)}:= X$ and 
\[
 L\p{(i)}_t=\sum_{s\leq t}(\Delta X_s)\p i=\int_{[0,t]\times\Rbb}p_i(x)N(\rmd s,\rmd x),\quad n\geq2.
\]
The process $L\p{(i)}$ is called $i$-th power jump asset of $X$. It can be seen that $L\p{(i)}$ is a L\'evy process with L\'evy measure given by the image measure $\nu\p{(i)}=\nu\circ p_i\p{-1}$, $i\geq1$.  Since $\nu\p{(i)}(p_j\p2)=\nu(p_{2j})<+\infty$, for $j\ge1$, $L\p{(i)}$ has finite moments of every order, for every $i\geq1$ and
\begin{equation}\label{eq:mom.Li}
\Ebb[L_1\p{(1)}]=\gm+\nu\big(1_{\{|p_1|>1\}}p_1\big),\qquad \Ebb[L_1\p{(i)}]=\nu(p_i),\quad  i \geq2.
\end{equation}
Hence, $X\p{(i)}_t:=L_t\p{(i)}-\Ebb[L_1\p{(i)}]t  $ is a square integrable martingale for $i\geq1$ and it is of finite variation for $i\geq2$. Furthermore, the identities
$\aPP{X\p{(i)}}{X\p{(j)}}_t=(\nu(p_{i+j})+\sig\p21_{\{i=j=1\}})t$ and
\begin{equation}\label{eq:teu.mar}
[X\p{(i)},X\p{(j)}]-\aPP{X\p{(i)}}{X\p{(j)}}=X\p{(i+j)},\quad i,j\geq	1,
\end{equation}
hold. This shows that the family $\Xscr:=\{X\p{(i)}, \, i\geq1\}$ is compensated-covariation stable. It is easy to recognize that the $i$-fold compensated-covariation process of the $i$-tuple of martingales $(X\p{(1)},\ldots,X\p{(1)})$ equals $X\p{(i)}$, for every $i\geq2$. The family $\Xscr$ is the family of Teugels martingales. 
Notice that the family $\Xscr$ can be seen as the \emph{compensated-covariation stable hull} of the process $X\p{(1)}$.

As a consequence of Proposition \ref{prop:rep.pol} applied to the family $\Xscr$, with $\al_1=\ldots=\al_N=1$, we can express the $N$-th power of $X\p{(1)}$ as follows: 
\begin{equation}\label{eq:prod.for.one.mart}
\begin{split}
(X\p{(1)}_t)\p N&=\sum_{i=1}\p N \binom{N}{i} \int_0\p t(X\p{(1)}_{s-})\p {N-i}\rmd  X\p {(i)}_s+\sum_{i=2}\p N \binom{N}{i}  \int_0\p t(X\p{(1)}_{s-})\p {N-i}\rmd\aPP{X\p{(i-1)}}{X\p{(1)}}_s
\\&=\sum_{i=1}\p N \binom{N}{i} \int_0\p t(X\p{(1)})_{s-}\p {N-i}\rmd  X\p {(i)}_s+\sum_{i=2}\p N \binom{N}{i}  \int_0\p t(X\p{(1)}_{s-})\p {N-i}(\nu(p_{i})+\sig\p21_{\{i=2\}})\rmd s.
\end{split}
\end{equation}
In the same way one can obtain formulas  for the power of each Teugels martingale $X\p{(n)}$, $n\geq2$: 
\begin{equation*}
\begin{split}
(X\p{(n)}_t)\p N
& =\sum_{i=1}\p N \binom{N}{i} \int_0\p t(X\p{(n)})_{s-}\p {N-i}\rmd  X\p {(ni)}_s+\sum_{i=2}\p N \binom{N}{i} \int_0\p t (X\p{(n)}_{s-})\p {N-i}\nu(p_{ni})\rmd s.
\end{split}
\end{equation*}

\subsection{Compensated-covariation stability of \texorpdfstring{$\Jscr_\rme$}{TEXT}}
In this subsection we show that, if $\Xscr=\{\X{\al},\ \al\in\Lm\} \subseteq\Hscr\p2$ is  compensated-covariation stable and $\aPP{X\p\al}{X\p\bt}$ is deterministic  for every $\al,\bt\in\Lm$, then  $\Jscr_\rme \subseteq\Hscr\p2$ and $\Jscr_\rme$ is  compensated-covariation stable.  
In this way  we can apply \eqref{eq:rep.pol}  to compute
products of elementary iterated integrals.
Notice that, for $Y, X\in\Jscr_\rme$, the process 
$\aPP{Y}{X}$ is not  deterministic  in general.

  We shall often need the following assumption. 
\begin{assumption}\label{ass:det.com.cov}
(i) The  family $\Xscr:=\{\X{\al},\ \al\in\Lm\}\subseteq\Hscr\p2$ is compensated-covariation stable. 

(ii)  $\Xscr$ consists of \emph{quasi-left continuous} martingales. 

(iii)  $\pb{\X{\al}}{\X{\bt}}$ denotes the \emph{continuous version} of the predictable covariation between $\X{\al}$ and $\X{\bt}$  and it is a
\emph{deterministic function of time} for every $\al,\bt\in\Lm$. 
\end{assumption} 
 The following lemma will be used to show that $\Jscr_\rme$ is compensated-covariation stable.  Recall that the notation $\Bb_T$ and $\Bb_T\p{\otimes m}$ was introduced in Definition \ref{elem.it.int.} (i). 
\begin{lemma}\label{lem:com.cov.it.int} 
Let $\Xscr$ satisfy Assumption \ref{ass:det.com.cov}.  For  $J_{m}^{\al_{1:m}}(F_{\otimes_m}), J_{n}^{\bt_{1:n}}(G_{\otimes_n}) \in\Jscr_\rme$  the processes $M$ and $N$ defined by
 \[
M_t:=\int_0\p tR\p{\,m,n}_{\,u-}H(u)\rmd \X{\gm}_{\,u},\qquad
N_t:=\int_0\p t\int_0\p {s-}R\p{\,m,n}_{\,u-}K(u)\rmd\aPP{\X{\dt}}{\X{\et}}_{\,u}H(s)\rmd \X{\gm}_{\,s}\,, \quad t\in[0,T],
\]
with $R\p{\,m,n}:=J_{m}^{\al_{1:m}}(F_{\otimes_m})J_{n}^{\bt_{1:n}}(G_{\otimes_n}),$
belong to $\bigoplus_{k=0}\p{m+n+1}\scr J_{k}$, for all  $\al_1,\ldots, \al_m;\ \bt_1,\ldots, \bt_{n};\ \gamma,\dt,\et\in\Lm$; 
for all   $F_{\otimes_m} \in \Bb_T\p {\otimes m}$ and $G_{\otimes_n} \in \Bb_T\p {\otimes n}$; for  all   $H, K \in \Bb_T$  and every $m, n\in\Nbb$.
\end{lemma}
\begin{proof}
For an arbitrary $m\in\Nbb$, we prove the lemma for every $n\in\Nbb$ by induction on $n$. 
If $n=0$, then $R\p{\,m,0}=J_{m}^{\al_{1:m}}(F_{\otimes_m})$ and so $M=J_{m+1}^{\al_{1:m},\gm}(F_{\otimes_m}\otimes H)$, which clearly belongs to 
the linear space $\bigoplus_{k=0}\p{m+1}\scr J_{k}$. We show that $N$ belongs to $\bigoplus_{k=0}\p{m+1}\scr J_{k}$  if $n=0$. 
The function $\wt F(t):=\int_0\p {t} K(u) \rmd\aPP{\X{\dt}}{\X{\et}}_{\,u}$ is deterministic, bounded and, by continuity of $\aPP{\X{\dt}}{\X{\et}}$, also continuous.
By integration by parts and continuity of $\wt F$, we get
\begin{equation}\label{eq:det.part0}
\begin{split}
\int_0\p t R\p{\,m,0}_{\,u-}K(u)\rmd\aPP{\X{\dt}}{\X{\et}}_{\,u}&=R\p{\,m,0}_{\,t}\wt F(t)-\int_0\p t\wt F(u)\rmd R\p{\,m,0}_{\,u}\\&=R\p{\,m,0}_{\,t}\wt 
F(t)-\int_0\p t J_{m-1}^{\al_{1:m-1}}(F_{\otimes_{m-1}})_{u-}\wt F(u)F_{m}(u)\rmd \X{\al_m}_{\,u}
\\&=R\p{\,m,0}_{\,t}\wt F(t)-J_{n}^{\al_{1:m}}(F_{\otimes_{m-1}}\otimes F_m\wt F)_{\,t}.
\end{split}
\end{equation}
Therefore 
\begin{equation}\label{eq:det.part}
\begin{split}
\int_0\p t\int_0\p {s-}R\p{\,m,0}_{\,u-}K(u)&\rmd\aPP{\X{\dt}}{\X{\et}}_{\,u}H(s)\rmd \X{\gm}_{\,s}\\&=
\int_0\p t R\p{\,m,0}_{\,s-}\wt F(s)H(s)\rmd \X{\gm}_{\,s}-J_{m+1}^{\al_{1:m},\gm}(F_{\otimes_{m-1}}\otimes F_m\wt F\otimes H)_{\,t}
\\&=J_{m+1}^{\al_{1:m},\gm}(F_{\otimes_{m}}\otimes H\wt F)_{\,t}-J_{m+1}^{\al_{1:m},\gm}(F_{\otimes_{m-1}}\otimes F_m\wt F\otimes H)_{\,t},
\end{split}
\end{equation}
which again belongs to $\bigoplus_{k=0}\p{m+1}  \scr J_{k}$ and this proves the basis of the induction. Now we assume that 
\[
M_t:=\int_0\p tR\p{\,m,j}_{\,u-}H(u)\rmd \X{\gm}_{\,u},\qquad
N_t:=\int_0\p t\int_0\p {s-}R\p{\,m,j}_{\,u-}K(u)\rmd\aPP{\X{\dt}}{\X{\et}}_{\,u}H(s)\rmd \X{\gm}_{\,s}
\]
belong to $\bigoplus_{k=0}\p{m+j+1} \scr J_{k} $, for every $j\leq n$, every $m$, every  $H,K \in \Bb_T$ and every $\gm,\dt,\et\in\Lm$ and prove the 
statement for $n+1$. By integration by parts we have 
\begin{align}
R_t\p{\,m,n+1}&=\int_0\p t R\p{\,m-1,n+1}_{u-}F_m(u)\rmd\X{\al_m}_{u}+\int_0\p t R\p{\,m,n}_{u-}G_{n+1}(u)\rmd\X{\bt_{n+1}}_{u}
\nonumber\\&\nonumber\qquad+
\int_0\p t R\p{\,m-1,n}_{u-}F_m(u)G_{n+1}(u)\rmd[\X{\al_m},\X{\bt_{n+1}}]_{u}
\\&=
\label{eq:term0}\int_0\p t R\p{\,m-1,n+1}_{u-}F_m(u)\rmd\X{\al_m}_{u}
\\&\label{eq:term1}\qquad+
\int_0\p t R\p{\,m,n}_{u-}G_{n+1}(u)\rmd\X{\bt_{n+1}}_{u}
\\&\label{eq:term2}\qquad+
\int_0\p t R\p{\,m-1,n}_{u-}F_m(u)G_{n+1}(u)\rmd\X{\al_m,\bt_{n+1}}_{u}
\\&\label{eq:term3}\qquad+
\int_0\p t R\p{\,m-1,n}_{u-}F_m(u)G_{n+1}(u)\rmd\aPP{\X{\al_m}}{\X{\bt_{n+1}}}_{u}\,,
\end{align}
where in \eqref{eq:term2} the compensated-covariation of $\X{\al_m}$ and $\X{\bt_{n+1}}$ appears.
Because of the induction hypothesis, \eqref{eq:term1} belongs to $\bigoplus_{k=0}\p{m+n+1}\scr J_{k}$. Therefore, $\int_0\p \cdot\eqref{eq:term1}_{s-}H(s)\rmd\X{\gm}_{s}$ belongs to $\bigoplus_{k=0}\p{m+n+2}\scr J_{k}$. We now discuss \eqref{eq:term2}. The family $\Xscr$ is compensated covariation stable. Therefore, the process $\X{\al_m,\bt_{m+1}}$ belongs to $\Xscr$.
 Obviously  $F_mG_{n+1}\in \Bb_T$. Hence, by the induction hypothesis, \eqref{eq:term2} belongs to $\bigoplus_{k=0}\p{m+n}\scr J_{k}$ and therefore 
$\int_0\p \cdot\eqref{eq:term2}_{s-}H(s)\rmd\X{\gm}_{s}$ belongs to $\bigoplus_{k=0}\p{m+n+1}\scr J_{k}\subseteq\bigoplus_{k=0}\p{m+n+2}\scr J_{k}$. From the induction 
hypothesis, it also immediately follows that $\int_0\p \cdot\eqref{eq:term3}_{s-}H(s)\rmd\X{\gm}_{s}$ belongs to 
$\bigoplus_{k=0}\p{m+n}\scr J_{k}\subseteq\bigoplus_{k=0}\p{m+n+2}\scr J_{k}$. We now show that $\int_0\p\cdot\eqref{eq:term0}_{s-}H(s)\rmd \X{\gm}_s$ belongs to 
$\bigoplus_{k=0}\p{m+n+2}\scr J_{k}$. For this it suffices to verify that \eqref{eq:term0} belongs to $\bigoplus_{k=0}\p{m+n+1}\scr J_{k}$. By integration by parts applied to $R\p{\,m-1,n+1}=J_{m-1}^{\al_{1:m-1}}(F_{\otimes_{m-1}})J_{n+1}^{\bt_{1:n+1}}(G_{\otimes_{n+1}})$ we get
\begin{align}
\nonumber\int_0\p t R\p{\,m-1,n+1}_{u-}F_m(u)\rmd\X{\al_m}_{u}
=&
\int_0\p t\int_0\p{s-} R\p{\,m-2,n+1}_{u-}F_{m-1}(u)\rmd\X{\al_{m-1}}_{u}F_m(s)\rmd\X{\al_m}_{s}
\\&+\nonumber
\int_0\p t\int_0\p{s-} R\p{\,m-1,n}_{u-}G_{n+1}(u)\rmd\X{\bt_{n+1}}_{u}F_m(s)\rmd\X{\al_m}_{s}
\\&+\nonumber
\int_0\p t\int_0\p{s-} R\p{\,m-2,n}_{u-}F_{m-1}(u)G_{n+1}(u)\rmd\X{\al_{m-1},\bt_{n+1}}_{u}F_m(s)\rmd\X{\al_m}_{s}
\\&+\nonumber
\int_0\p t\int_0\p{s-} R\p{\,m-2,n}_{u-}F_{m-1}(u)G_{n+1}(u)\rmd\aPP{\X{\al_{m-1}}}{\X{\bt_{n+1}}}_{u}F_m(s)\rmd\X{\al_m}_{s}\,.
\end{align}
In other words, because of the induction hypothesis, we can rewrite this expression as 
\[
\int_0\p t R\p{\,m-1,n+1}_{u-}F_m(u)\rmd\X{\al_m}_{u}=\int_0\p t\int_0\p{s-} R\p{\,m-2,n+1}_{u-}F_{m-1}(u)\rmd\X{\al_{m-1}}_{u}F_m(s)\rmd\X{\al_m}_{s}+V_t
\]
where $V\in\bigoplus_{k=0}\p{m+n+1}\scr J_{k}$. Iterating this procedure $m-1$ times, we get
\[
\begin{split}
&\int_0\p t R\p{\,m-1,n+1}_{u-}F_m(u)\rmd\X{\al_m}_{u}=\\&U_t+\int_0\p t\int_0\p{t_{m-1}-}\cdots\int_0\p{t_2-} J_{n+1}^{\bt_{1:n+1}}( G_{\otimes_{n+1}} )_{t_1-}\;F_{1}(t_1)\rmd\X{\al_{1}}_{t_1}\ldots F_{m-1}(t_{m-1})\rmd\X{\al_{m-1}}_{t_{m-1}} F_m(t_{m})\rmd\X{\al_m}_{t_{m}},
\end{split}\]
where $U\in\bigoplus_{k=0}\p{m+n+1}\scr J_{k}$. Using the definition of  the elementary iterated integral, we get
\[ 
\int_0\p t R\p{\,m-1,n+1}_{u-}F_m(u)\rmd\X{\al_m}_{u}=J_{n+m+1}^{\bt_{1:n+1},\al_{1:m}}(G_{\otimes_{n+1}}\otimes F_{\otimes_{m}})_t+U_t
\]
which belongs to $\bigoplus_{k=0}\p{m+n+1}\scr J_{k}$ and this  completes the proof for $M$. To prove the statement for $N$ we have to verify that
\[  
N_t:=\int_0\p t\int_0\p {s-}R\p{\,m,n+1}_{\,u-}K(u)\rmd\aPP{\X{\dt}}{\X{\et}}_{\,u}H(s)\rmd \X{\gm}_{\,s}
\]
belongs to $\bigoplus_{k=0}\p{m+n+2}\scr J_{k}$, for every $m$, every $H,K \in \Bb_T$ and every $\gm,\dt,\et\in\Lm$. From the previous step, 
$\int_0\p\cdot R_{u-}\p{\,m,n+1}H(u)\rmd X\p\gm_{u}\in\bigoplus_{k=0}\p{m+n+2}\scr J_{k}$, for every $m$, every   $H \in \Bb_T$ and every $\gm\in\Lm$. 
Hence, with similar computations as in \eqref{eq:det.part0} and \eqref{eq:det.part}, using the representation of $R\p{m,n+1}$ obtained in 
\eqref{eq:term0} -- \eqref{eq:term3}, we deduce the claim from the induction hypothesis. The proof of the lemma is complete. 
\end{proof}
 We can now prove that $\Jscr_\rme$  
is compensated-covariation stable. 
\begin{theorem}\label{thm:com.cov.st.it.int}
Let $\Xscr$ satisfy Assumption \ref{ass:det.com.cov}.  Then the family $\Jscr_\rme$ of the elementary iterated integrals generated by $\Xscr$ is compensated-covariation stable and any $M\in\Jscr_\rme$ is quasi-left continuous.
\end{theorem}
\begin{proof}
Let $M,N\in\Jscr_\rme$. By linearity we can assume without loss of generality that 
$ M=J_{m}^{\al_{1:m}}(F_{\otimes_{m}})$ and $ N=J_{n}^{\bt_{1:n}}(G_{\otimes_{n}})$.
Because of  \eqref{eq:def.sim.mul.int} we  have
\begin{equation}\label{eq:com.cov.it.int.2}
\begin{split}
[M,N]_t-\aPP{M}{N}_t&=
[J_{m}^{\al_{1:m}}(F_{\otimes_{m}}),J_{n}^{\bt_{1:n}}(G_{\otimes_{n}})]_t
-
\aPP{J_{m}^{\al_{1:m}}(F_{\otimes_{m}})}{J_{n}^{ \bt_{1:n}}(G_{\otimes_{n}})}_t
\\&=
\int_0\p t J_{m-1}^{\al_{1:m-1}}(F_{\otimes_{m-1}})_{u-}J_{n-1}^{\bt_{1
:n-1}}(G_{\otimes_{n-1}})_{u-}F_{m}(u)G_{n}(u)\rmd\X{\al_m,\bt_n}_u\,.
\end{split}
\end{equation}
Since by assumption $\Xscr$ is compensated-covariation stable, we have $\X{\al_m,\bt_n}\in\Xscr$. Hence, $F_{n}G_{m}$ being bounded,  the compensated-covariation stability of $\Jscr_\rme$ immediately follows from Lemma \ref{lem:com.cov.it.int}. To see the quasi-left continuity of $M$ it is enough to observe that
\begin{equation}\label{eq:pb.M}
\aPP{M}{M}=\big(J_{m-1}^{\al_{1:m-1}}(F_{\otimes_{m-1}})_-F_m)\p 2\cdot\aPP{X\p{\al_m}}{X\p{\al_m}}
\end{equation}
and hence $\aPP{M}{M}$ is continuous by the continuity of $\aPP{X\p{\al_m}}{X\p{\al_m}}$. The quasi-left continuity of $M$ is a consequence of the last statement of \cite[Theorem I.4.2]{JS00}. The proof of the theorem is now complete.
\end{proof}

 As a next step we study the $m$-fold compensated covariation process  built from elements of  $\Jscr_\rme$. For $J_{m}^{\al_{1:m}}(F_{\otimes_{m}}),  \, J_{n}^{\bt_{1:n}}(G_{\otimes{n}}) \in\Jscr_\rme$ we denote by 
\begin{equation*}
\mathbf J\p{\,\mathbf\al_{1:m},\mathbf\bt_{1:n}}_{m,n}( F_{\otimes_{m}}\otimes G_{\otimes{n}}):=[J_{m}^{\al_{1:m}}(F_{\otimes_m}),J_{n}^{ \bt_{1:n}}(G_{\otimes_n})]-\aPP{J_{m}^{\al_{1:m}}(F_{\otimes_m})}{J_{n}^{ \bt_{1:n}}(G_{\otimes_n})}
\end{equation*} 
their  compensated-covariation process. For any $j\in\Nbb$ we  introduce the notation 
\[
F\p j_{\otimes_m}:=F_1\p{j}\otimes\cdots\otimes F_m\p{j}.
\]
Because of Theorem \ref{thm:com.cov.st.it.int}, for $j=1,\ldots,N$, $\al_{1:m_j}\p j\in\Lm\p{m_j}$ and $J_{m_j}\p{\al_{1:m_j}\p j}(F_{\otimes_{n_j}}\p j)\in\Jscr_\rme$, we can inductively define 
\begin{equation}\label{eq:com.cov.it.in}
\begin{split}
\mathbf J\p{\,\mathbf\al_{1:m_1}\p1,\cdots,\al\p{N}_{1:m_N}}_{m_1,\ldots,m_n}&( F\p{1}_{\otimes_{m_1}}\otimes\cdots\otimes F\p{N}_{\otimes{m_N}}):=
\\&
\Big[\mathbf J\p{\,\al_{1:m_1}\p1,\ldots,\al\p{n}_{1:m_N-1}}_{m_1,\ldots,m_{N}-1}( F\p{1}_{\otimes_{m_1}}\otimes\cdots\otimes F\p{{N-1}}_{\otimes_{m_N-1}}),J_{m_N}^{\al\p N_{1:m_N}}(F\p N_{\otimes_{m_N}})\Big]
-\\&\qquad\qquad
\Big\langle\mathbf J\p{\,\al_{1:m_1}\p1,\ldots,\al\p{n}_{1:m_n-1}}_{m_1,\ldots,m_{N}-1}( F\p{1}_{\otimes_{m_1}}\otimes\cdots\otimes F\p{{N-1}}_{\otimes_{m_N-1}}),J_{m_N}^{\al\p N_{1:m_N}}(F\p N_{\otimes_{m_N}})\Big\rangle\,.
\end{split}
\end{equation}
\begin{proposition}\label{prop:com.cov.it.int.expr} Let $\Xscr$ satisfy Assumption \ref{ass:det.com.cov}.
Then 
\begin{equation}\label{eq:rep.comcov.it.int}
\mathbf J\p{\,\mathbf\al_{1:m_1}\p1,\ldots,\mathbf\al\p{N}_{1:m_N}}_{m_1,\ldots,m_N}(F\p{1}_{\otimes_{m_1}}\otimes\cdots\otimes F\p N_{\otimes_{m_N}})=\Bigg(\prod_{j=1}\p NJ_{m_j-1}^{\al\p j_{1:m_{j}-1}}( F\p j_{\otimes_{m_{j}-1}})_{-}F_{m_j}\p{j} \Bigg)\cdot\X{\al\p1_{m_1},\ldots,\al\p N_{m_N}}\,.
\end{equation}
\end{proposition}
\begin{proof}
 For $N=2$  the statement coincides with \eqref{eq:com.cov.it.int.2}. Then, the formula immediately follows by induction from \eqref{eq:com.cov.it.in} and Theorem \ref{thm:com.cov.st.it.int}.
\end{proof}
Notice that the integrator appearing on the right-hand side of \eqref{eq:rep.comcov.it.int} is the $m$-fold compensated-covariation process of the ordered $m$-tuple of martingales $(\X{\al_{m_1}\p1},\ldots,\X{\al_{m_N}\p N} )$.

\subsection{A product formula for elementary iterated integrals} 
We now exploit \eqref{eq:rep.pol} for elementary iterated integrals generated by a compensated-covariation stable family of martingales. Observe that to prove Proposition 
\ref{prop:rep.pol} for a family of martingales $\Xscr=\{X\p\al,\ \al\in\Lm\}$, it is not needed that the predictable covariation $\aPP{X\p{\al_1}}{X\p{\al_2}}$ is deterministic but it is 
sufficient that the family $\Xscr$ is  compensated-covariation stable  and consists of quasi-left continuous martingales.   This is important because for $X,Y\in\Jscr_\rme$, the process $\aPP{X}{Y}$ is continuous but not deterministic, in general (cf.\  \eqref{eq:pb.M}).
\begin{theorem}\label{thm:prod.for.it.int} Let $\Xscr$ satisfy Assumption \ref{ass:det.com.cov}. 
 Then, for  $J_{m_j}^{\al\p{ j}_{1:m_j}}(F\p j_{\otimes_{m_j}}) \in\Jscr_\rme$,   $j=1,\ldots,N$,  $N\geq2$, we have 
\begin{equation}\label{eq:prod.for.it.int}
\begin{split}
&\prod_{j=1}\p NJ_{m_j}^{\al\p j_{1:m_j}}(F\p j_{\otimes_{m_j}})\\&=\quad
\sum_{i=1}\p N\sum_{{1\leq j_1<\ldots<j_i\leq N}}\bigg(\prod_{\begin{subarray}{c}k=1\\k\neq j_1,\ldots,j_i\end{subarray}}\p N
J_{m_k}^{\al\p k_{1:m_k}}(F\p k_{\otimes_{m_k}})_-\prod_{\ell=1}\p i J\p{\al\p {j_\ell}_{1:{m_{j_\ell}-1}}}_{{m_{j_\ell}-1}}(
F\p{j_\ell}_{\otimes{m_{j_\ell}-1}})_-F_{m_{j_\ell}}\p{j_\ell}\bigg)\cdot X\p{\al_{m_{j_1}}\p{{j_1}},\ldots,\al\p{{j_i}}_{m_{j_i}}}
\\&\quad+
\sum_{i=2}\p N\sum_{{1\leq j_1<\ldots<j_i\leq N}}\bigg(\prod_{\begin{subarray}{c}k=1\\k\neq j_1,\ldots,j_i\end{subarray}}\p N J_{m_k}^{\al\p k_{1:{m_k}}}(
F\p k_{\otimes_{m_k}})_-\prod_{\ell=1}\p i J\p{\al\p {j_\ell}_{1:{m_{j_\ell}-1}}}_{{m_{j_\ell}-1}}(F\p{j_\ell}_{\otimes_{m_{j_\ell}-1}})_-F_{m_{j_\ell}}
\p{j_\ell}\bigg)\cdot\aPP{X\p{\al_{m_{j_1}}\p{{j_1}},\ldots,\al\p{{j_{i-1}}}_{m_{j_{i-1}}}}}{X\p{\al\p{{j_i}}_{m_{j_i}}}}\,.
\end{split}
\end{equation}
\end{theorem}
\begin{proof}
 We have $\Jscr_\rme \subseteq\Hscr\p2$. Moreover, from Theorem \ref{thm:com.cov.st.it.int}, $M\in\Jscr_\rme$ is quasi-left continuous and   $\Jscr_\rme$ is  compensated-covariation stable. So we can introduce the 
compensated-covariation processes \eqref{eq:com.cov.it.in} and by Proposition \ref{prop:rep.pol} 
\[
\begin{split}
\prod_{j=1}\p N &J_{m_j}^{\al\p j_{1:m_j}}(F\p j_{\otimes_{m_j}})=\\&
\sum_{i=1}\p N\sum_{{1\leq j_1<\ldots<j_i\leq N}}\bigg(\prod_{\begin{subarray}{c}k=1\\k\neq j_1,\ldots,j_i\end{subarray}}\p N J_{m_k}^{\al\p k_{1:{m_k}}}\big(F\p k_{\otimes_{m_k}}\big)_-\bigg)
\cdot\mathbf J\p{\al_{1:m_{j_1}}\p{j_1},\ldots,\mathbf\al\p{j_i}_{1:m_{j_i}}}_{m_{j_1},\ldots,m_{j_i}}\big(F_{\otimes_{m_{j_1}}}\p{j_1}\otimes\cdots\otimes F\p{j_i}_{\otimes_{m_{j_i}}}\big)
\\&+
\sum_{i=2}\p N\sum_{{1\leq j_1<\ldots<j_i\leq N}}\bigg(\prod_{\begin{subarray}{c}k=1\\k\neq j_1,\ldots,j_i\end{subarray}}\p N J_{m_k}^{\al\p k_{1:{m_k}}}\big(F\p k_{\otimes_{m_k}}\big)_-\bigg)\cdot
\\&
\hspace{4cm}\cdot\bigg\langle\mathbf J\p{\al_{1:m_{j_1}}\p{j_1},\ldots,\al\p{j_{i-1}}_{1:m_{j_{i-1}}}}_{m_{j_1},\ldots,m_{j_{i-1}}}\big(F_{\otimes{m_{j_1}}}\p{j_1}\otimes\cdots\otimes 
F\p{j_{i-1}}_{\otimes_{m_{j_{i-1}}}}\big),J_{m_{j_i}}^{\al\p {j_i}_{1:{m_{j_i}}}}\big(F\p {j_i}_{\otimes_{m_{j_i}}}\big)\Big\rangle\,.
\end{split}
\]
The statement follows from Proposition \ref{prop:com.cov.it.int.expr}.
\end{proof}
In the next step we obtain a recursive formula for computing moments of products from $\Jscr_\rme$.  To ensure their existence we make the following assumption:
\begin{assumption}\label{ass:mom.mart}
Each martingale in $\Xscr$ has finite moments of every order.
\end{assumption}
Let $\Xscr$ satisfy Assumption \ref{ass:det.com.cov}. We observe that then, according to \cite[Corollary 4.5]{DTE16}, $\Xscr$ fulfils Assumption \ref{ass:mom.mart} if there exists $\bt\in\Lm$ such that $\aPP{\X{\bt}}{\X{\bt}}_t<\aPP{\X{\bt}}{\X{\bt}}_T$, $t<T$.

Before we come to the moment formula we need the following technical result. For the proof see \cite[Lemma 5.5]{DTE16}.
\begin{lemma}\label{lem:int.prod.X}
Let  $\Xscr:=\{X\p{\al},\ \al\in \Lm\} \subseteq\Hscr\p{\,2}$ satisfy Assumption \ref{ass:mom.mart} and let $A$ be a deterministic process of finite variation. We define  the process $K$ by
\begin{equation*}
K:=\prod_{i=1}\p NX\p{\al_i}_{-},\quad \al_i\in \Lm,\quad i=1,\ldots,N.
\end{equation*}
Then the process $K\cdot A$ is of integrable variation.
\end{lemma}
Let $A$ and $K$ be as in Lemma \ref{lem:int.prod.X}. Then by $\Ebb[K]\cdot A$ we denote the integral with respect to $A$ of the $A$-integrable function $t\mapsto\Ebb[K_t]$.

 Now we are ready to state and prove a recursive moment formula for products of elementary iterated integrals.
\begin{corollary}\label{cor:mom.it.int}
Let $\Xscr$ satisfy Assumptions \ref{ass:det.com.cov} and \ref{ass:mom.mart}. Then, for  $J_{m_j}^{\al\p{ j}_{1:m_j}}(F\p j_{\otimes_{m_j}}) \in\Jscr_\rme$,   $j=1,\ldots,N$,  $N\geq2$,  the following formula holds:
\begin{equation}\label{eq:mom.it.int}
\begin{split}
&\Ebb\Bigg[\prod_{j=1}\p NJ_{m_j}^{\al\p j_{1:{m_j}}}(F\p j_{\otimes_{m_j}})\Bigg]=
\\&
\quad\sum_{i=2}\p N\sum_{{1\leq j_1<\ldots<j_i\leq N}}\Ebb\left[\prod_{\begin{subarray}{c}k=1\\k\neq j_1,\ldots,j_i\end{subarray}}\p N 
J_{m_k}^{\al\p k_{1:{m_k}}}(F\p k_{\otimes_{m_k}})_-\prod_{\ell=1}\p i J\p{\al\p {j_\ell}_{1:{m_{j_\ell-1}}}}_{{m_{j_\ell-1}}}
(F\p{j_\ell}_{\otimes_{{m_{j_\ell}-1}}})_-F_{m_{j_\ell}}\p{j_\ell}\right]\cdot
\aPP{X\p{\al_{m_{j_1}}\p{{j_1}},\ldots,\al\p{{j_{i-1}}}_{m_{j_{i-1}}}}}{\X{\al\p{{j_i}}_{m_{j_i}}}}.
\end{split}
\end{equation}
\end{corollary}
\begin{proof}
Because of Lemma \ref{lem:mom.it.int}, if $\Xscr$ satisfies Assumption \ref{ass:mom.mart}, the family $\Jscr_\rme$ of elementary iterated integrals also does. Therefore the left-hand side of \eqref{eq:prod.for.it.int} is integrable. Furthermore, because  $\aPP{\X{\al}}{\X{\al}}$ is deterministic, for every  $\al \in\Lm$, we conclude by  Lemma \ref{lem:int.prod.X} that the integrands of the stochastic integrals in the first term on the right hand side of \eqref{eq:prod.for.it.int} belong to $\Lrm\p2(\X{\al})$, for every $\al\in\Lm$. Therefore each summand in the first term on the right-hand side of \eqref{eq:prod.for.it.int} belongs to $\Hscr\p2$. Analogously,  we have  by Lemma \ref{lem:int.prod.X} and Assumption \ref{ass:det.com.cov} that the second term on the right-hand side of \eqref{eq:prod.for.it.int} is integrable. We can therefore consider the expectation and apply the theorem of Fubini to conclude the proof. 
\end{proof}

 Notice that  by  taking $m_j=1$ and $F_1\p j=1$, $j=1,\ldots,N$  in \eqref{eq:mom.it.int} one can \emph{recursively} compute  expressions like  $ \Ebb[\prod_{j=1}\p N X\p{\al_j}]$. Indeed, in this special case, the second product on the right-hand side of \eqref{eq:mom.it.int} is identically equal to one, while the first product consists of a number of factors which is strictly smaller than $N$.  For example, if $X$ is a Brownian motion,  \eqref{eq:mom.it.int} corresponds to taking the expectation in \eqref{eq:prod.for.one.mart.bm}
 and  we get
$\Ebb[X\p {2N+1}_t]=0$, $N\geq0$, and    
\[
\Ebb[X_t\p {2N}]=(2N-1)!!(\sqrt{t})\p{2N},\quad N\geq2,
\]
where $N!!$ denotes the \emph{double factorial} of $N$. 

If $X$ is a compensated Poisson process with parameter $\lm$, then by  taking the expectation in \eqref{eq:prod.for.one.mart.poi.pr}  (or by simplifying \eqref{eq:mom.it.int})
we obtain
\[
\Ebb[X_t\p N]=\lm \sum_{i=2}\p N\binom{N}{i} \int_0\p t\Ebb[(X_{s})]\p {N-i}\rmd s
\] 
which is the formula in \cite[Proposition 3.3.4 and Example 3.3.5]{PT11}.  The resulting  polynomials 
\[\begin{split}
&\Ebb[X_t\p2]=\lm t,\quad  \Ebb[X_t\p3]=\lm t ,\quad   \Ebb[X_t\p4]=3\lm\p2 t\p2+\lm t,\quad   \Ebb[X_t\p5]=10\lm\p2 t\p2+\lm t\ldots,
\end{split}
\]
also called \emph{centred Touchard polynomials}.   

 If, more generally, $X$ is a L\'evy process (see  Subsection \ref{subs:exa}) with finite moments of every order, we  derive  from Corollary \ref{cor:mom.it.int} the \emph{recursive} formula \eqref{thm:mom.lev} below for the central moments of $X$. This formula allows to compute the moments of a L\'evy process recursively without  taking derivatives of its characteristic function. 
\begin{theorem}\label{thm:mom.lev} Let $X$ be a L\'evy process with characteristic triplet $(\gm,\sig\p2,\nu)$, such that $X$ has finite moments of every order. Let us define the centred process $X\p{(1)}$ by $X\p{(1)}_t:=X_t-\Ebb[X_t]$. Then, for every $N\geq1$, we have
\begin{equation} \label{eq:mom.lev.cen}
\Ebb[(X\p{(1)}_t)\p N]=\sum_{i=2}\p N \binom{N}{i}  (\nu(p_{i})+\sig\p21_{\{i=2\}}) \int_0^t \Ebb[(X\p{(1)}_{s-})\p {N-i}]\rmd s.
\end{equation}
Consequently, for $N\geq1$, the non-central moments of $X$ are given by
\begin{equation}\label{eq:mom.lev.ncen}
\begin{split}
\Ebb[X\p N_t]
=\sum_{i=1}\p N\binom{i}{k}\big((\gm+\nu(1_{\{|p_1|>1\}}))t\big)\p{N-i}\sum_{k=1}\p i \binom{i}{k}(\nu(p_{k})+\sig\p21_{\{k=2\}}) \int_0^t \Ebb[(X\p{(1)}_{s-})\p {i-k}]\rmd s.
\end{split}
\end{equation}
\end{theorem}
\begin{proof}
The family $\Xscr:=\{X\p{(i)}, \, i\geq1\}$ defined by \eqref{eq:teu.mar} is compensated-covariation stable, consists of quasi-left continuous martingales and is such that $\aPP{X\p{(i)}}{X\p{(j)}}$ is deterministic.  Hence, $\Xscr$ satisfies Assumptions \ref{ass:det.com.cov}. Furthermore, because of the last paragraph in Subsection \ref{subs:exa}, $\Xscr$ satisfies also Assumption \ref{ass:mom.mart}. Hence, from the identity $(X\p{(1)}_t)\p N=\prod_{j=1}\p N J_1\p{1}(1)_t$, it follows that \eqref{eq:mom.lev.cen} is a special case of Corollary \eqref{cor:mom.it.int}. To see \eqref{eq:mom.lev.ncen}, we observe that the identity $X_t=X\p{(1)}_t+\Ebb[X_t]$ holds. Thus, 
$
\Ebb[X\p N]=\sum_{i=1}\p N \binom{N}{i}\Ebb[(X\p{(1)}_t)\p i]\Ebb[X_t]\p{N-i}
$
and \eqref{eq:mom.lev.ncen} immediately follows from \eqref{eq:mom.lev.cen} because of \eqref{eq:mom.Li}. The proof is complete.
\end{proof}

\section{Iterated integrals with respect to L\'evy processes}\label{sec:Lev.pr}
In this section $X$ is a L\'evy process with characteristic triplet $(\gm,\sig\p2,\nu),$ and we denote by $N$ the jump measure of $X$ and by $\tilde N:=N-\nu\otimes\lm_+$ the compensated Poisson random measure. 
For  $\al\in L\p2(\nu)$  and $t\geq0$ we use the notation 
$$
1_{ [0, t]}\al \ast  \tilde N:=\int_{[0,t]\times \Rbb } \al(x)\tilde N(\rmd s,\rmd x).
$$
 We set $\mu := \sigma^2 \delta_0 + \nu$, where $\delta_0$ is the Dirac measure concentrated in zero. For $ \al \in L\p2(\mu)$, we define
\equal \label{poisson-integral}
   X\p\al:=\al(0) W^\sigma +(1_{ [0, \cdot]\times (\Rbb \setminus \{ 0\}) }\al) \ast  \tilde N.
\tionl

We recall that, if $\al\in L\p2(\mu)$, then the process $X\p\al$ has the following properties:
\bigskip 

(i) $(X^\al,\Fbb)$ is a L\'evy process with characteristic triplet  $(-\int_{\{|x|>1\}} \al(x) \nu(dx), \al(0)^2\sigma^2, \nu \circ \al^{-1} )$.

(ii) $X^\al\in\Hscr\p2(\Fbb)$  and $\aPP{X\p\al}{X\p\al}_t=t\mu(\al\p2)$.

(iii) $\Delta X^\al=\al(\Delta X)1_{\{\Delta X\neq0\}} $ and hence $\Delta X^\al$ is bounded, if $\al$ is bounded.

(iv) Let $\bt\in L\p2(\mu)$. Then $\aPP{X\p\al}{X\p\bt}=0$ if and only if $\mu(\al\bt)=0$.
\bigskip 

For a system $\Lm \subseteq L\p2(\mu)$, we put

\equal \label{X-lambda}
\Xscr_\Lm:=\{X\p \al,\,\, \al\in \Lm\}.
\tionl

\begin{assumption}\label{ass:good.sys} 
Let $\Lm$ be a set of real-valued functions with the following properties: \\
(i) $\Lm\subseteq L\p{\,1}(\mu)\cap L\p{\,2}(\mu)$; \\
(ii) $\Lm$ is total in $L\p{\,2}(\mu)$; \\
(iii) $\Lm$ is stable under multiplication, and $1_{\Rbb\setminus\{0\}}\al\in\Lm$ whenever $\al\in\Lm$; \\
(iv) $\Lm$ is a system of bounded functions. 
\end{assumption}
We observe that a system $\Lm$ satisfying Assumption \ref{ass:good.sys} always exists: Obviously, we can choose 
$\Lm:=\{\al=c1_{\{0\}}+1_{(a,\,b]},\ a,b\in\Rbb:a<b,\ 0\notin[a,b];\ c\in\Rbb\}\cup\{0\}$ as an example.

\begin{proposition}  
Let  $\Lm$ satisfy Assumption \ref{ass:good.sys}. Then 
$\Xscr_\Lm \subseteq \Hscr\p{\,2}(\Fbb)$. For $\al_1,\ldots,\al_m  \in \Lm$ 
 the   compensated-covariation process $X^{\al_{1:m}}$, $m\ge 2$,  has the following form:
\equal \label{co-co stable processes}
    X^{\al_{1:m}} = \Bigl (1_{[0, \cdot]\times(\Rbb \setminus \{ 0\})} \prod_{k=1}^m\al_k  \Bigr ) \ast  \tilde N.  
\tionl
Moreover, \\
{\rm(i)} $\Fbb\p{\Xscr_\Lm}=\Fbb\p X$;\\
{\rm(ii)} $\Ebb[\exp(\lm|X_{\,t}|)]<+\infty$ for every $X\in\Xscr_\Lm$, $\lm>0$, $t\in[0,T]$, \\
{\rm (iii)} $\pb XY$ is deterministic for every $X,Y\in\Xscr_\Lm$, \\
{\rm (iv)} $\Xscr_\Lm$ possesses the CRP with respect to $\Fbb\p X$.
\end{proposition}

\begin{proof} For the ``Moreover'' part we refer to  \cite[Proposition 6.4 and Proposition 6.5]{DTE16}. 
To show \eqref{co-co stable processes},  we notice that
 $$ [X^{\al_1}, X^{\al_2}]_t= \al_1(0) \al_2(0) \sigma^2 t +  \int_{[0, t]\times\Rbb \setminus \{ 0\}} \al_1(x)\al_2(x)\,N(\rmd s,\rmd x) $$
and 
$$ \langle X^{\al_1}, X^{\al_2}\rangle_t= \al_1(0)\al_2(0) \sigma^2 t +   \nu ( \al_1\al_2  )t.  $$
Then, since $\al_1\al_2\in L\p1(\mu)\cap L\p2(\mu)$, we get  $ X^{\al_{1:2}} =  \Bigl (1_{[0, \cdot]\times(\Rbb \setminus \{ 0\})} \al_1\al_2  \Bigr ) \ast \tilde   N , $ and   \eqref{co-co stable processes}
follows   from  \eqref{eq:it.com.cov} by  induction.
\end{proof}

\subsection{Products of elementary iterated integrals}
The aim of this subsection is to deduce a \emph{product formula} for elements from  $\Jscr_\rme$,  where we assume that $\Jscr_\rme$ is  generated by  $\Xscr_\Lm$ and  $\Lm\subseteq L\p2(\mu)$ satisfies Assumption \ref{ass:good.sys}.

By Definition \ref{elem.it.int.}, an elementary iterated integral generated by  the martingales $ X\p{\al_1},\ldots,X\p{\al_m}$, $ {\al_1},\ldots,{\al_m}\in\Lm$, is given by
\begin{equation*}
J_{m}\p{\al_1,\ldots, \al_{m}}(F_{\otimes_m})_{\,t}
:=\int_0^t J_{m-1}^{\al_1,\ldots, \al_{m-1}}( F_{\otimes_{m-1}})_{u-}\,F_{m}(u)\, \rmd {\X{\al_{m}}_u},\qquad t\in [0,T], \quad m\geq 1\,.
\end{equation*}
Our aim is to determine  a product formula for
\[
\prod_{j=1}^N J_{m_j}\p{\al_1\p j,\ldots, \al_{m_j}\p j}(F_{\otimes_{m_j}}\p j),\quad  N\in\Nbb,
\] 
(given in equation \eqref{product-rule}  below) by iterating formula \eqref{eq:prod.for.it.int} until the inner elementary  iterated integrals appearing as integrands will reduce to deterministic functions. 
 In this way it will be  also possible to determine a formula for the moment of products of elementary iterated integrals which  generalizes  the isometry relation \eqref{eq:IRmix.el}.
For this goal we need to introduce some combinatoric rules and notations. 

First of all  we fix $N$, that is, the number of factors of the product, and then we fix $m_1,\ldots,m_N$, that is, the order of each factor. We then order all the functions $F_{k}\p j \in \Bb_T$, $k=1,\ldots,m_j$, $j=1,\dots,N$,  consecutively: Set
\equa    \overline m_0&:=&0,   \\ 
\overline m_j&:=&m_1+\ldots+m_j, \quad j =1,\dots, N.
\tion 
 We then define
\[
 (F_{\overline m_{j-1}+1},\ldots,F_{\overline m_j}):= (F_{1}\p j,\ldots,F_{m_j}\p j),\quad  j=1,\ldots,N,
\]
and set
\[
F_{\otimes_{\overline m_{j-1}+1:\overline m_j}}:=F_{\overline m_{j-1}+1}\otimes\cdots\otimes F_{\overline m_j},\quad j=1,\ldots,N.
\]
Analogously,  we define 
\[
 (\al_{\overline m_{j-1}+1},\ldots,\al_{\overline m_j}):= (\al_{1}\p j,\ldots,\al_{m_j}\p j),\quad j=1,\ldots,N.
\]
 Like in \eqref{product-rho}, we put $ \rho\p{\al_{\overline m_{j-1}+1:\overline m_j}} := \rho\p{\al_{\overline m_{j-1}+1}}\otimes \ldots \otimes \rho\p{\al_{\overline m_j}}$, 
 where  \[
\rmd\rho\p{\al_{\overline m_{j-1}+1}:\overline m_j} (t_1,\ldots,t_{m_j})=\bigg(\prod_{i=\overline m_{j-1}+1}\p{\overline m_j} \nu\big(\al_i\p2\big)\bigg)\rmd\lm( t_1,\ldots,t_{m_j}).
\]

With this notation we get: 
\equal \label{index-renaming}
  \prod_{j=1}^N J_{m_j}\p{\al_1\p j,\ldots, \al_{m_j}\p j}(F_{\otimes_{m_j}}\p j)
= \prod_{j=1}^N J_{m_j}\p{\al_{\overline m_{j-1}+1:\overline m_j}}(F_{\otimes_{\overline m_{j-1}+1:\overline m_j}}).
\tionl

Notice that a stochastic integral with respect to $X^\al$   is in fact the sum of a stochastic integral with respect to  $ \al(0)W^\sigma$ and one with respect to $(1_{[0, \cdot]\times(\Rbb \setminus \{ 0\})}\al) \ast  \tilde N$. For convenience we  shall write  these two integrals instead of the integral  with respect to $X^\al$ to recognize whether we are integrating with respect to the Brownian part or with respect to the jump part. To this aim, let $B \subseteq \{1,2,\ldots, \overline m_N\}$ denote the set of
indices for which we integrate with respect to the Brownian part of  $X^{\al_k}$, $k=1,\ldots,\overline m_N$. We define 
$$\al^B_k:=  \left \{ \begin{array}{ll} \al_k 1_{\{0\}} & \text{ if } k \in B,   \\ 
               \al_k   1_{\Rbb \setminus \{ 0\}} &\text{ if } k \notin B. 
\end{array}      \right .$$  
Then  
\begin{equation}\label{eq:trans.prod.bm.int}
 \prod_{j=1}^N J_{m_j}\p{\al_{\overline m_{j-1}+1:\overline m_j}}(F_{\otimes_{\overline m_{j-1}+1:\overline m_j}})=\sum_{B \subseteq \{1,2,\ldots, \overline m_N\}}  
 \prod_{j=1}^N J_{m_j}\p{\al^B_{\overline m_{j-1}+1:\overline m_j}}(F_{\otimes_{\overline m_{j-1}+1:\overline m_j}}),
\end{equation}
which implies that, to obtain the  multiplication formula, it suffices to transform  the product on the right-hand side of \eqref{eq:trans.prod.bm.int} into a sum of iterated integrals, for any  $B \subseteq \{1,2,\ldots, \overline m_N\}$.
We observe that the extreme cases $B=\{1,2,\ldots, \overline m_N\}$ and $B=\emptyset$ correspond to the case of integration  with respect to the Brownian part only and  with respect to the jump part only, respectively.
Let us consider, as an illustrating example, the case   $N=2$, $m_1=m_2=1$. Here we have  $B \in \{\emptyset,\{1\},\{2\},\{1,2\}\}$ and 
\begin{gather*}
\al\p{\{1\}}_1=\al_11_{\{0\}},\quad\al \p{\{1\}}_2=\al_21_{\Rbb\setminus\{0\}},\quad\al\p{\{2\}}_1=\al_11_{\Rbb\setminus\{0\}},\quad \al\p{\{2\}}_2=\al_21_{\{0\}},\\\\ \al\p{\{1,2\}}_1=\al_11_{\{0\}},\quad \al\p{\{1,2\}}_2=\al_21_{\{0\}},\quad \al\p{\{\emptyset\}}_1=\al_11_{\Rbb\setminus\{0\}},\quad \al\p{\{\emptyset\}}_2=\al_21_{\Rbb\setminus\{0\}}.
\end{gather*}
Notice that \eqref{eq:prod.for.it.int} holds also for $\prod_{j=1}^N J_{m_j}^{\al^B_{\overline m_{j-1}+1: \overline m_j}}(F_{\otimes_{\overline m_{j-1}+1: \overline m_j}})$ and reads as 
\begin{equation}\label{levy-product}
\begin{split}
&\prod_{j=1}^N J_{m_j}^{\al^B_{\overline m_{j-1}+1: \overline m_j}}(F_{\otimes_{\overline m_{j-1}+1: \overline m_j}})\\
&\quad=\sum_{r=1}\p N\sum_{{1\leq j_1<\ldots<j_r\leq N}}\Bigg\{\Bigg( \prod_{\begin{subarray}{c}k=1\\k\neq j_1,\ldots,j_i\end{subarray}}\p N J_{m_k}^{\al^B_{\overline m_{k-1}+1: \overline m_k}}(F_{\otimes_{\overline m_{k-1}+1: \overline m_k}})_- \Bigg ) \\
&\quad\quad\times
 \Bigg ( \prod_{\ell=1}\p r   J_{m_{j_\ell-1}}^{\al^B_{ \overline m_{j_{\ell-1}} +1: \overline m_{j_\ell-1}}} (F_{\otimes_{\overline m_{j_\ell-1}+1: \overline m_{j_{\ell}-1}}})_-   \Bigg )  \Bigg (  \prod_{q=1}\p r  F_{\overline m_{j_q}}  \Bigg ) 
 \Bigg\}\cdot\X{ \al^B_{\overline m_{j_1}},\ldots,\al^B_{\overline m_{j_r}}}\\
 &\quad\quad\quad\quad\quad+
\sum_{r=2}\p N\sum_{{1\leq j_1<\ldots<j_r\leq N}}\Bigg\{ \Bigg( \prod_{\begin{subarray}{c}k=1\\k\neq j_1,\ldots,j_i\end{subarray}}\p N J_{m_k}^{\al^B_{\overline m_{k-1}+1: \overline m_k}}(
F_{\otimes_{\overline m_{k-1}+1: \overline m_k}})_-\Bigg )\\
&\quad\quad \times \Bigg( \prod_{\ell=1}\p r   J_{m_{j_\ell-1}}^{\al^B_{ \overline m_{j_{\ell-1}} +1: \overline m_{j_\ell-1}}} (F_{\otimes_{\overline m_{j_\ell-1}+1: \overline m_{j_{\ell}-1}}})_-\Bigg ) \Bigg( \prod_{q=1}\p r   F_{\overline m_{j_q}} \Bigg )
  \Bigg\}\cdot\aPP{\X{\al^B_{\overline m_{j_1}},\ldots, \al^B_{\overline m_{j_{r-1}}} }}{\X{\al^B_{\overline m_{j_r}}}}\,.
\end{split}
\end{equation}

We observe that the integrands on the right-hand of \eqref{levy-product} consist of products of elementary iterated integrals. These elementary iterated integrals  are either of the same  order $m_k$ as the ones appearing in the product 
on the left-hand side of \eqref{levy-product} or of the diminished order $m_{j_\ell-1}$. We will repeatedly apply  \eqref{levy-product} to the { \it integrands} on the right-hand side of \eqref{levy-product} until we get  integrands  that consist of iterated integrals of order zero, that is, they are equal to one. 
\bigskip \\
{\bf The integrators in \eqref{levy-product}.}  
In order to determine the exact structure of the outcome of the procedure explained above, we want to specify which  integrators 
$\rmd \X{ \al^B_{\overline m_{j_1}},\ldots,\al^B_{\overline m_{j_r}}}$ can occur in \eqref{levy-product}.
If $r=1$, we have 
\equa \begin{array}{lll} 
\X{ \al^B_{\overline m_{j_1}}}&= \al_{\overline m_{j_1}}(0) W^\sigma, \quad & \text{ if }\quad \overline m_{j_1} \in B, \\
\X{ \al^B_{\overline m_{j_1}}}&= (1_{\Rbb \setminus \{ 0\}}1_{[0, \cdot]}\al_{\overline m_{j_1}}) \ast  \tilde N, \quad & \text{ if }\quad
 \overline m_{j_1}\in B^c =\{1,2,\ldots,\overline m_N\} \setminus B .\end{array} \tion
From \eqref{co-co stable processes} we conclude that  for $r \ge 2$ 
$$\N{ \al^B_{\overline m_{j_1}},\ldots,\al^B_{\overline m_{j_r}}}=  \N{1_{\Rbb \setminus\{0\}}\prod_{k=1}^r \al_{\overline m_{j_k}} } 
=    \left \{ \begin{array}{ll}  \Bigl (1_{[0, \cdot]\times(\Rbb \setminus \{ 0\})} \prod_{k=1}^r \al_{\overline m_{j_k} } \Bigr ) \ast  \tilde N, 
& \text{ if }\quad  \{ \overline m_{j_1},\ldots,\overline m_{j_r} \} \subseteq B^c, \\ &\\
 0, &\text{ otherwise.}
\end{array}      \right .  $$
Furthermore, it holds 
$$\aPP{\N{\al^B_{\overline m_{j_1}} }}{\X{\al^B_{\overline m_{j_2}}}}_t=   t \, \mu( \al^B_{\overline m_{j_1}}\al^B_{\overline m_{j_2}})
= \left \{ \begin{array}{ll}  
   t \,\nu   (  \al_{\overline m_{j_1}}\al_{\overline m_{j_2}}  ),  &\text{ if }  \{\overline m_{j_1}, \overline m_{j_2} \} \subseteq B^c, \\\\
   t \sigma^2 \al_{\overline m_{j_1}}\!\!(0) \, \al_{\overline m_{j_2}}\!\!(0),  &\text{ if } \{\overline m_{j_1}, \overline m_{j_2} \} \subseteq B,  \\\\
   0, &   \text{ otherwise, }  \end{array}  \\ \right .   $$
and for  $r \ge 3$,
$$\aPP{\N{\al^B_{\overline m_{j_1}},\ldots,  \al^B_{\overline m_{j_{r-1}}} }}{\X{\al^B_{\overline m_{j_r}}}}_t = 
 \left \{ \begin{array}{ll}    t \,\nu \Bigl  ( \prod_{k=1}^r \al_{\overline m_{j_k}}  \Bigr ), \quad & \text{ if }    
 \{\overline m_{j_1},\ldots,  \overline m_{j_r} \} \subseteq B^c,  \\ \\
  $0$, & \text{otherwise}. \end{array}      \right . $$
{\bf The integrators in the iteration steps.} On the right-hand side of \eqref{levy-product}, integrals with respect to martingales and with respect to {\it deterministic}  point-bracket processes appear.
If we apply \eqref{levy-product} to the integrands, we will get ``mixed'' iterated integrals where the integrators are both martingales
and point brackets. We will use the superscript  $i$ where $i=1$ stands for ``martingale''  and $i=0$ for point bracket.
 If  $S\subseteq\{ 1, \ldots, \overline m_N\}$ we define   
 \begin{equation*}
 \al_{(S)} := \prod_{\ell \in S}  \al_{\ell},   \quad    \widehat \al_{(S)} := 1_{\Rbb \setminus\{0\}}\al_{(S)}. 
\end{equation*}
For  $ i \in \{ 0,1\}$, we will write,  summarizing  the analysis of the integrators above, 
\equal \label{Y-in-detail}
\rmd Y^{ \al_{(S),i} }_t := \left \{ \begin{array}{ll} \nu(\al_{(S)}) \rmd t
& \text{ if }\quad i=0,  \, \,S \subseteq B^c, \,\, |S|\ge 2,     \\ \\
\sigma^2\al_{(S)}(0)  \rmd t  & \text{ if }\quad i=0, \,\,S \subseteq B,  \,\, |S|= 2,  \\\\
               \al_{(S)}(0)\rmd W^\sigma_t   & \text{ if }\quad  i=1, \,\,S\subseteq B, \,\, |S|= 1, \\\\
           \rmd  X^{\widehat \al_{(S)}}_t                      & \text{ if }\quad  i=1, \,\, S\subseteq B^c , \,\, |S| \ge 1,  \\\\
           0, &\text{ otherwise}.
\end{array}      \right .
\tionl
\bigskip \\
{\bf Algorithm to build  identification rules for the integrands.} 
In \eqref{levy-product} the function   $t \mapsto \prod_{q=1}\p r   F_{\overline m_{j_q}}(t)  $ appears on the right hand side. Since we want to use the tensor product  $\bigotimes_{i=1}^{ N} F_{\otimes \overline m_{i-1}+1:\overline m_i}$ as integrand in the final formula, we need to identify $t:= t_{\overline m_{j_1}}=\ldots= t_{\overline m_{j_r}}$  to get this ordinary  product from the tensor product. These ordinary products arise from each application of \eqref{levy-product}, hence we want to  derive an identification  rule  that describes  which of the variables $t_1,\ldots,t_{\overline m_N}$ need to be identified in each step.
Set 
$$[\overline m_N] :=\{1,2,\ldots, \overline m_N\}. $$
 $\Pi^{}(m_1,\ldots,m_N)$  will  denote the set of all those partitions  ${\bf s}= (S_1,\ldots,S_k)$  of the set  $[\overline m_N]$   which can be  built using the following backward induction from $k$ till $1:$\\
\underline{First Step: $k$} \, Choose a non-empty subset
$$ S_k \subseteq \{ \overline m_1, \ldots, \overline m_N\}. $$
\underline{Step: $k-l+1 \to k-l$} \, Assume that  $S_k, S_{k-1}, \ldots, S_{k-l+1}$ (for some $l =1,\ldots,k-1$) have been chosen. \\
 By $I_1,..., I_N$  we denote  the sets which contain exactly those indices  of the $F_j$'s  which are used in the {\it same } elementary iterated integral on the left-hand side of \eqref{index-renaming}, that is, 
$$I_1 := \{1, \ldots,m_1\}  \quad \text{  and  }  \quad I_r := \{\overline m_{r-1}+1, \ldots, \overline m_r \}, \quad r=2,\ldots,N.$$ 
Then  $S_{k-l}$  denotes   a non-empty subset of the set consisting of  the largest elements of $I_1,\ldots,I_N$  {\it which have not yet been chosen for $S_k, S_{k-1}, \ldots, S_{k-l+1}$ already}:
If
\begin{equation*}
   L^{k-l}_r  = \sup\{ I_r \setminus  ( S_k \cup S_{k-1} \cup \ldots\cup S_{k-l+1})\}, \quad   r=1,\ldots,N,
\end{equation*}
 where we use the convention $\sup \emptyset= -\infty$, then 
$$S_{k-l}  \subseteq    \{ L^{k-l}_r:   L^{k-l}_r\neq -\infty, \,\, r=1,\ldots,N \}.$$
Notice that we require  that  $\bigcup_{\ell=1}^k S_\ell = [\overline m_N]$.    
Any such ${\bf s}= (S_1,\ldots,S_k) \in \Pi^{}(m_1,\ldots,m_N)$ will be called an {\bf identification rule}.  We summarize this procedure in the 
following definition.
\begin{definition}  \label{partitions} Let $\Pi^{}(m_1,\ldots,m_N)$ be the set of all partitions  ${\bf s}= (S_1,\ldots,S_{|{\bf s}|})$  of $\{1,2,\ldots, \overline m_N\}$  (here $|{\bf s}|$ stands for the number of sets which belong to ${\bf s}$), where $N \le |{\bf s}| \le \overline m_N$, such that \\
{\rm(i)}   each $S_\ell$ contains at most one element of   $\{\overline m_{r-1}+1, \ldots, \overline m_r \}$, that is,   \\
 $$ |S_\ell \cap  \{\overline m_{r-1}+1, \ldots, \overline m_r \}| \le 1 \quad \text{ for all} \quad  \ell=1,...,|{\bf s}|, \,\, r=1,\ldots,N,$$ 
{\rm(ii)}  the elements of each  $\{\overline m_{r-1}+1, \ldots, \overline m_r \}$ appear ordered within  ${\bf s}$, that is, 
  if  for  $1\le \ell < k \le |{\bf s}|$  it holds $a \in  S_\ell \cap  \{\overline m_{r-1}+1, \ldots, \overline m_r \}$  and  $b \in  S_k \cap  \{\overline m_{r-1}+1, \ldots, \overline m_r \}, $ then $a<b.$
\end{definition}

For a function
$(u_1,\ldots,u_{\overline m_N}) \mapsto  H(u_1,\ldots,u_{\overline m_N})$ and  ${\bf s}= (S_1,\ldots,S_k)$, we define the following identification of variables: 
\equal \label{s-rule}
&& H_{{\bf s}}(t_1,\ldots,t_k) \,\, \text{ is derived from  } \,\, H(u_1,\ldots,u_{\overline m_N}) \,\,   \text{ by  replacing   all  }  \,\,  u_j   \,\,\text{  with } \,\, j \in  S_r \,\, \notag \\  
&&  \text{  by the same  } \,\,  t_r \text{  for } \,\,   r=1,\ldots,k.    
 \tionl   
 
As already indicated above  the  identification rules   ${\bf s}= (S_1,\ldots,S_k)$   will be used to describe the integrands  which result from the  iteration of  \eqref{levy-product}: In the first step,  which is  \eqref{levy-product} itself,  the factor $ \prod_{q=1}\p r  F_{\overline m_{j_q}}$   appears. Notice that   
$$ \prod_{q=1}\p r  F_{\overline m_{j_q}}   = \prod_{\ell \in S_k}  F_{\ell}   \quad \text{ provided that } \quad S_k = \{\overline m_{j_1}, \ldots, \overline m_{j_r}\}.$$
At the same time, the iterated integrals with order $m_{j_1}, \ldots,  m_{j_r},$ which appear on the left-hand side of  \eqref{levy-product},  have its order diminished by $1$
on the right-hand side.
If we apply \eqref{levy-product}  to the integrands of   the right-hand side,  then the product  $ \prod_{\ell \in S_{k-1}}  F_{\ell} $ will appear, where 
$$S_{k-1}\subseteq (\{ \overline m_1, \ldots, \overline m_N\}\setminus S_k) \cup  \{\overline m_{j_1}-1, \ldots,  \overline m_{j_r}-1 \}.   $$ 
We repeatedly apply \eqref{levy-product} and, finally, we  have $ \prod_{r=1}^k   \prod_{j \in S_r} F_j (t_r)$. 
This product we get from  the tensor product   $  \left ( \bigotimes_{j=1}\p N F_{\otimes \overline m_{j-1}+1:\overline m_{j+1}}  \right )(u_1,\ldots, u_{\overline m_N})= F_1(u_1) \times \ldots\times F_{\overline m_N}(u_{\overline m_N}) $ applying the identification rule ${\bf s} =(S_1,\ldots, S_k)$:
 $$\left (  \bigotimes_{j=1}\p N F_{\otimes \overline m_{j-1}+1:\overline m_j} \right )_{\bf s}(t_1,\ldots,t_{k}) = \prod_{r=1}^k   \prod_{j \in S_r} F_j (t_r) .$$ 
Identification rules  similar to  the above ones appear  also in Peccati and Taqqu \cite[Chapter 7]{PT11}  and Last et al. \cite{LPST14} (see also the references therein).
 \bigskip \\
{\bf The set  $I_{{\bf s}}$.}  We still need to pay attention to the fact that each application of \eqref{levy-product}  produces  integrals  with respect to both
compensated-covariation processes and  point brackets. 
For  any identification rule ${\bf s}= (S_1,\ldots,S_{|{\bf s}|}) \in   \Pi^{}(m_1,\ldots,m_N) $   we define  the sets 
$$I_{{\bf s}}=    \{  {\bf i} := (i_1,\ldots, i_{|{\bf s}|}):     i_\ell \in  i(S_\ell), \,\,  \ell=1,\ldots, |{\bf s}| \},$$
where 
\equa
i(S_\ell) := \left \{ \begin{array}{ll} \{ 0\},  & \text{ if }   |S_\ell |  = 2, \,\, S_\ell \subseteq   B,   \\ 
\{ 1\},    & \text{ if }   |S_\ell|  = 1,    \\
   \{ 0, 1\} ,             & \text{ if } |S_\ell|  \ge  2 ,  \,\, S_\ell \subseteq   B^c, \\
    \emptyset, &\text{ otherwise}.
\end{array}      \right .  
\tion
 So, in view of \eqref{Y-in-detail}, if for example $ |S_\ell |  = 2$ and $S_\ell \subseteq   B$, then we have just one integral, hence $ i(S_\ell)$ contains one element. This integral is with respect to $\rmd t$,
which we indicate here by $ i(S_\ell)=\{ 0\}$. If  $|S_\ell|  = 1,$ then we also  have just one integral but  the integrator is  a martingale, so we use $i(S_\ell)=\{ 1\}$, and so on.

We denote by  $M_t^{|{\bf s}|}$ the simplex $M_t^{\,m}$ (cf.\ \eqref{eq:def.Mn}) with $m:=|{\bf s}|$. Then we get 
\equal \label{sum_with_zeros}
&&  \prod_{j=1}^N J_{m_j}^{\al_{\overline m_{j-1}+1: \overline m_j}}(F_{ \otimes{\overline m_{j-1}+1: \overline m_j}})_t  \notag\\ 
&=& \!\!\!\!\sum_{B \subseteq [\overline m_N]} \,\,  \sum_{{\bf s} \in  \Pi^{}(m_1,\ldots,m_N)} \,\,\  \sum_{ {\bf i}  \in  I_{{\bf s}}}  
 \int_{M_t^{|{\bf s}|}} \left ( \bigotimes_{j=1}\p N F_{\otimes \overline m_{j-1}+1:\overline m_j} \right )_{\bf s}(t_1,\ldots,t_{|{\bf s}|}) 
  \rmd Y_{t_1}^{\al^{B}_{(S_1),i_1}} \dots \rmd Y^{\al^{B}_{(S_{|{\bf s}|}), i_{|{\bf s}|}}}_{t_{|{\bf s}|}}. 
 \tionl

{\bf Avoiding zeros in the summation.}
The right-hand side of \eqref{sum_with_zeros} contains many terms which are zero: From  \eqref{Y-in-detail} we see that this happens, for example, if there is a set $S_\ell$ containing elements from  both, $B$ and $B^c$. Another case where the integral is zero is if there is  a set $S_\ell \subseteq B$ with $|S_\ell| \ge 3$. We will exclude these cases by defining
\equal  \label{s-no-zero}
\Pi^{}_{\le 2, \ge1 }(B, B^c; m_1,\ldots,m_N) &:=& \big \{ {\bf s} =(S_1,\ldots,S_{|{\bf s}|})  \in  \Pi^{}(m_1,\ldots,m_N):
   \forall \ell =1,\ldots,|{\bf s}|  \text{ it holds } \notag \\
&&\quad\quad\quad\quad  (S_\ell \cap B =S_\ell  \text{ and  }  |S_\ell| \le 2 ) \,\, \text{ or } \,\, S_\ell \cap B^c =S_\ell  \big \}.
 \tionl
Hence we have derived the following theorem:

\begin{theorem}\label{thm:prod.rule}   Let $X$ be a L\'evy process and suppose that  $\Lm\subseteq L\p2(\mu)$ satisfies
Assumption \ref{ass:good.sys}. For $J_{m_j}^{\al_{\overline m_{j-1}+1: \overline m_j}}(F\p j) \in \Jscr_\rme$, $j=1,\ldots,N$, 
 with $F\p j:=F_{ \otimes{\overline m_{j-1}+1: \overline m_j}}   \in \Bb_T^{m_j}$, the following product formula holds: 

\equal \label{product-rule}
\hspace*{-5em}&& \hspace*{-2em}\prod_{j=1}^N J_{m_j}^{\al_{\overline m_{j-1}+1: \overline m_j}}(F\p j)_t  \notag\\ 
\hspace*{-1em}&=&\!\!\!\!\sum_{B \subseteq [\overline m_N]} \,\,  \sum_{{\bf s} \in  \Pi_{\le 2, \ge 1}^{}(B,B^c; m_1,\ldots,m_N)} \,\,\  \sum_{ {\bf i}  \in  I_{{\bf s}}}  
 \int_{M_t^{|{\bf s}|}} \bigg( \bigotimes_{j=1}^{ N} F\p j \bigg )_{\bf s}(t_1,\ldots,t_{|{\bf s}|}) 
  \rmd Y_{t_1}^{\al^{B}_{(S_1),i_1}} \ldots \rmd Y^{\al^{B}_{(S_{|{\bf s}|}), i_{|{\bf s}|}}}_{t_{|{\bf s}|}}.
	\tionl
\end{theorem}
Let us now denote by $\lm\p{|\bf s|}$  the $|\bf s|$-dimensional Lebesgue measure restricted to the simplex $M\p{|\bf s|}_t$ (cf.\ \eqref{eq:def.Mn}). Our aim is to compute the expectation  of \eqref{product-rule}. We shall call the resulting relation \emph{moment formula}. Before we define

\equal  \label{pi2}
\Pi^{}_{= 2, \ge2 }(B, B^c; m_1,\ldots,m_N) &:=& \big \{ {\bf s} =(S_1,\ldots,S_{|{\bf s}|})  \in  \Pi^{}(m_1,\ldots,m_N):
   \forall \ell =1,\ldots,|{\bf s}|  \text{ it holds } \notag \\
&&(S_\ell \cap B =S_\ell  \text{ and  }  |S_\ell| = 2 ) \,\, \text{ or } \,\, S_\ell \cap B^c =S_\ell \text{ and  }  |S_\ell| \ge2  \big \}.
 \tionl

\begin{corollary}\label{cor:moment.formula}  Let $X$ be a L\'evy process and suppose that  $\Lm\subseteq L\p2(\mu)$ satisfies
Assumption \ref{ass:good.sys}. Then for $J_{m_j}^{\al_{\overline m_{j-1}+1: \overline m_j}}(F^j ) \in \Jscr_\rme$, $j=1,\ldots,N$,
where $F^j = F_{ \otimes{\overline m_{j-1}+1: \overline m_j}}   \in \Bb_T^{m_j}$, the moment formula 
\equal \label{moment-formula}
 \EE  \bigg [\prod_{j=1}^N J_{m_j}^{\al_{\overline m_{j-1}+1: \overline m_j}}(F^j)_t\bigg ] 
    &=& \!\!\!\!\sum_{B \subseteq [\overline m_N]} \,\,  \sum_{{\bf s} \in  \Pi_{= 2, \ge 2}^{}(B,B^c; m_1,\ldots,m_N)} \,\,\ 
 \bigg (  \prod_{\begin{subarray}{c}\ell=1\\ S_\ell \subseteq B^c\end{subarray}}^{|{\bf s}|} \nu(\alpha_{(S_\ell)})  \bigg )  
 \bigg (   \prod_{\begin{subarray}{c}q=1\\ S_q \subseteq B\end{subarray}}^{|{\bf s}|}   \bigg (\alpha_{(S_q)}(0) \sigma^2 \bigg ) \bigg ) \notag  \\
&&  \quad \quad \quad  \times \int_{M_t^{|{\bf s}|}} \bigg (\bigotimes_{i=1}^{N} F^j \bigg )_{\bf s}(t_1,\ldots,t_{|{\bf s}|})  
    \rmd  \lambda^{|{\bf s}|} (t_1,\ldots,t_{|{\bf s}|}) 
\tionl holds.
\end{corollary}
\begin{proof} First we observe that,  according to \eqref{Y-in-detail}, the process  $Y^{\al^{B}_{(S_\ell),i_\ell}}$ can be deterministic only if $|S_\ell| \ge2$  and $i_\ell=0$, and in this case, its explicit expression is
\[
 Y_{t}^{\al^{B}_{(S_\ell), 0}}=\begin{cases}\nu(\al_{(S_\ell)}) \, t,&\quad\textnormal{if }\ S_\ell \subseteq B^c,\\
\sigma^2\al_{(S)}(0)\, t,&\quad\textnormal{if }\ S_\ell \subseteq B\ \text{ and }\ |S_\ell |=2.
\end{cases}
\]
If $Y^{\al^{B}_{(S_\ell),i_\ell}}$ is random, then  $i_\ell=1$ and  $Y^{\al^{B}_{(S_\ell),1}}\in\Hscr\p{\,2}$.
We now take the expectation in \eqref{product-rule} and notice that, if at least one of the  $Y^{\al^{B}_{(S_\ell),i_\ell}}$ is random, we have
$$ \EE  \left[\int_{M_t^{|{\bf s}|}} \bigg ( \bigotimes_{j=1}^{ N} F^j \bigg )_{\bf s}(t_1,\ldots,t_{|{\bf s}|}) 
  \rmd Y_{t_1}^{\al^{B}_{(S_1),i_1}} \ldots \rmd Y^{\al^{B}_{(S_{|{\bf s}|}), i_{|{\bf s}|}}}_{t_{|{\bf s}|}}\right] =0 $$
 because the integrand is bounded.  This implies that
 only  integrals  with ${\bf s} \in  \Pi_{= 2, \ge 2}^{}(B,B^c; m_1,\ldots,m_N)$ will appear,  and since  $i_\ell=0$ for all $\ell$ there is no sum over  ${\bf i}  \in  I_{{\bf s}},$ so that we get \eqref{moment-formula}. The proof is complete.
\end{proof} 

 \begin{example}  \label{calculation} We illustrate the moment formula \eqref{moment-formula} with some examples. Choose $\alpha_1,\ldots ,\al_5 $ and let  $X^{\alpha_1},\ldots,X^{\alpha_5}$ be given by \eqref{poisson-integral}. 
\begin{enumerate}
\item For $N=3$
we assume that  $m_1=1,$ $m_2=1$ and  $m_3=2.$ We want to  apply \eqref{moment-formula} to compute
$$\EE \big[J_1^{\alpha_1}(F^1)_t  J_1^{\alpha_2}(F^2)_t   J_2^{\alpha_3,\alpha_4 }(F^3)_t\big]. $$ 

By Definition \ref{partitions} we get  the identification rules  $${\bf s_1}= (\{1,3 \},\{ 2,4 \}) \quad \text{ and }\quad  {\bf s_2}=(\{1,4 \},\{2,3\}).$$
 In this case $\overline m_3=4$. The subsets of $[\overline m_3]=\{1,2,3,4\}$ which are relevant for the construction of $B$ are the elements of the family $\Bscr:=\{\{1,3\},\{1,4\},\{2,3\},\{2,4\},\{1,2,3,4\},\emptyset\}$, since $\Pi_{=2,\ge2}(B,B\p c;1,1,2)=\emptyset$ for $B\notin\Bscr$. We now compute $\Pi_{=2,\ge2}(B,B\p c;1,1,2)$ for $B\in\Bscr$. For shortness, we write $\Pi_{=2,\ge2}(B,B\p c)$. We have 
 $$\Pi_{=2,\geq2}(\{1,3\},\{2,4\})=\Pi_{=2,\geq2}(\{2,4\},\{1,3\})={\bf s_1},$$ 
 $$\Pi_{=2,\geq2}(\{1,4\},\{2,3\})=\Pi_{=2,\geq2}(\{2,3\},\{1,4\})={\bf s_2},$$ and finally we have 
 $$\Pi_{=2,\ge2}([\overline m_3],\emptyset)=\Pi_{=2,\ge2}(\emptyset,[\overline m_3])=\{{\bf s_1},{\bf s_2}\}.$$ We stress that $\Pi_{=2,\ge2}([\overline m_3],\emptyset)$ corresponds to the case  where  only integration with respect to the Brownian part is considered, while $\Pi_{=2,\ge2}(\emptyset,[\overline m_3])$ corresponds to the opposite case where only integration with respect to the jump part is considered. Hence,
\equa
\EE \big[J_1^{\alpha_1}(F^1)_t  J_1^{\alpha_2}(F^2)_t   J_2^{\alpha_3,\alpha_4 }(F^3)_t\big] 
&=& C_1 \int_0^t \int_0^r  F^1(u)  F^2(r)  F^3(u,r)  du dr \\&& + C_2 \int_0^t \int_0^r  F^1(r)  F^2(u)  F^3(u,r)  du dr,
\tion 
where 
$$C_1= \nu(\alpha_1\alpha_3)   \nu(\alpha_2\alpha_4)  +   \nu(\alpha_1\alpha_3) (\alpha_2\alpha_4)(0)\sigma^2+   \nu(\alpha_2\alpha_4)(\alpha_1\alpha_3)(0)\sigma^2    +   (\alpha_1\alpha_2\alpha_3\alpha_4)(0)\sigma^4 $$
and 
$$C_2= \nu(\alpha_2\alpha_3)   \nu(\alpha_1\alpha_4)  +   \nu(\alpha_2\alpha_3) (\alpha_1\alpha_4)(0)\sigma^2+    \nu(\alpha_1\alpha_4)(\alpha_2\alpha_3)(0)\sigma^2   + (\alpha_1\alpha_2\alpha_3\alpha_4)(0) \sigma^4 .$$
\item For $N=3$ with   $m_1=1,$ $m_2=1$ and  $m_3=3$ we notice that  $\Pi^{}_{= 2, \ge2 }(B, B^c; m_1,m_2,m_3) =\emptyset $, for every $B\subseteq\{1,2,3,4,5\}$. This can be seen because 
by Definition \ref{partitions} (i) each element of $\{\overline m_{2+1}, \overline m_{2+2}, \overline m_{2+3}\}$ has to be in a different partition set. So we should have a least 3 partition sets, each with (at least) 2 elements. This is impossible as $m_1+m_2+m_3=5.$ Hence,
\equa
\EE \big[J_1^{\alpha_1}(F^1)_t  J_1^{\alpha_2}(F^2)_t   J_2^{\alpha_3,\alpha_4,\alpha_5}(F^3)_t \big]=0.
\tion
It is easy to see that this observation can be generalized as follows: If
there exists a $m_j$ such that $m_j >  m_1+...+m_{j-1}+m_{j+1} +...+m_N$ then 
$$ \EE  \bigg [\prod_{j=1}^N J_{m_j}^{\al_{\overline m_{j-1}+1: \overline m_j}}(F^j)_t\bigg ] =0.$$ 
\end{enumerate}
\end{example}

Moment formulas similar to \eqref{moment-formula} have also been obtained by Peccati and Taqqu in \cite[Corollary 7.4.1]{PT11} and  by  Last et al.\ in \ 
\cite[Theorem 1]{LPST14} for multiple integrals (similar to the multiple integrals introduced by It\^o in \cite{I56}) generated by a compensated 
Poisson random measure. The proofs in \cite{PT11}  and  \cite{LPST14} rely on  \emph{Mecke's Formula} (see \cite{M67}),  which is not applicable if the 
L\'evy process has a Gaussian part. Here we use \eqref{levy-product}  instead of \emph{Mecke's Formula}. We  stress that for the case  without Gaussian part 
the moment formula  in \cite[Theorem 1]{LPST14}
only requires an $L\p1$-condition. This is due to the fact that for $\sig=0$ the multiple integrals can be considered pathwise.

\begin{remark}[Product and moment formula for linear combinations of elementary functions]\label{rem:ext.lin.com}
 Let  $\Escr^{m_j}_T $ denote the linear subspace of  $L\p2(M\p{m_j}_T,\rho\p{\al_{\overline m_{j-1}+1:\overline m_j}})$ generated by  $F_{\otimes_{m_j}} \in \Bb_T\p {\otimes m_j}$ restricted to ${M}_{\,T}^{\, m_j}$. The elementary iterated  integrals can be uniquely linearly extended to elements of $\Escr^{m_j}_T$. 
Hence, if $F\p j\in\Escr^{m_j}_T$ has the representation
\[
F\p j=\sum_{k=1}\p {M_j}F\p{j,k}_{\otimes\overline m_{j-1}+1:\overline m_j},\quad j=1,\ldots,N,
\]
then
\[
\prod_{j=1}\p N J_{m_j}\p{\al_{\overline m_{j-1}+1}: \al_{\overline m_j}}(F\p j)=\prod_{j=1}\p N\Big(\sum_{k=1}\p{M_j}J_{m_j}\p{\al_{\overline m_{j-1}+1}: \al_{\overline m_j}}(F\p{j,k}_{\otimes\overline m_{j-1}+1:\overline m_j})\Big).
\]
Using the formula
\[
\prod_{j=1}\p N\Big(\sum_{k=1}\p{M_j}a\p{j,k}\Big)=\sum_{j_1=1}\p{M_1}\ldots\sum_{j_N=1}\p{M_N}a\p{1,j_1}a\p{2,j_2}\cdots a\p{N,j_N},
\]
valid for real numbers $a\p{j,k}$, $k=1,\ldots,N_j$; $j=1,\ldots,N$, the multi-linearity of the tensor product in Theorem \ref{thm:prod.rule} and the linearity of the identification rule, it is clear that \eqref{product-rule} extends to this more general case.
From Corollary \ref{cor:moment.formula} it is now clear that the moment formula \eqref{moment-formula}  holds also for   $F\p j\in\Escr^{m_j}_T$.

\end{remark}

\subsection{Extensions of the product and moment formula} 

 For practical applications Assumption \ref{ass:good.sys} might not be desirable because  it  requires that $\Lm$ is  stable under multiplication.  
For example, if  $\Lm$ is an orthogonal basis of $L\p 2(\mu)$, in general,  it fails to be  stable under multiplication. However, $\Lm$ being an orthogonal 
basis is especially interesting because then $\Xscr_\Lm$, defined in \eqref{X-lambda}, consists of countably many orthogonal martingales, and it possesses the CRP with respect to $\Fbb\p X$ in the simpler form given in Remark \ref{rem:alt.rep.crp}.  Therefore, in the present subsection, we extend \eqref{product-rule} and \eqref{moment-formula}  to products of iterated integrals generated by $\Xscr_\Lm$, where $\Lm$ is an arbitrary system in $\Lm\subseteq\bigcap_{p\geq1}L\p p(\mu)$.

To begin with, we prove the following technical lemma.  We denote by $\Bb_\Rbb$ the space of bounded measurable functions on $\Rbb$.
\begin{lemma}\label{lem:bou.fun.app}
Let $\al\in L\p2(\mu)\cap L\p p(\mu)$, $p>2$, and let $\Lm:=L\p1(\mu)\cap\Bb_\Rbb$. Then there exists $(\al\p k)_{k\geq1}\subseteq\Lm$ such that $\al\p k\longrightarrow \al$ in $L\p q(\mu)$  as $k\rightarrow+\infty$ for every $2 \le q\leq p$.

\end{lemma}
\begin{proof}
For $\al\in L\p2(\mu)\cap L\p p(\mu)$ we define $\et:=\al1_{\{|\al|\leq1\}}$ and $\gm:=\al1_{\{|\al|>1\}}$. Then, $\al=\et+\gm,$  where $\et\in L\p2(\mu)$ and $\gm\in L\p p(\mu)$. Because $\Lm$ is dense in $L\p q(\mu)$, for every $q>1$, there exist $(\et\p k)_{k\geq1}$ and $(\gm\p k)_{k\geq1}$ in $\Lm$ such that $\et\p k\longrightarrow \et$ in $L\p2(\mu)$ and $\gm\p k\longrightarrow \gm$ in $L\p p(\mu)$ as $k\rightarrow+\infty$, respectively.  Since $\mu(\{|\al|>1\})<+\infty$, we have  $\gm\p k\longrightarrow \gm$   also in $L\p2(\mu)$. Notice that, because $\et$ is bounded by one, we can also assume that $(\et\p k)_{k\geq1}$ is uniformly bounded by one.  But then  $\et\p k\longrightarrow \et$   in $L\p p(\mu)$.  A sequence $(\al\p k)_{k\geq1}\subseteq\Lm$ converging to $\al$ in $L\p2(\mu)$ and in $L\p p(\mu)$ is given, for example, by $\al\p k:=\et\p k1_{\{|\al|\leq1\}}+\gm\p k1_{\{|\al|>1\}}\in\Lm$, $k\geq1$.
Finally, if $(\al\p k)_{k\geq1}$ converges to $\al$ in $L\p 2(\mu)$ and $L\p p(\mu)$, then it also converges to $\al$ in $L\p q(\mu)$ for every $2 \le q\leq p$ .
\end{proof}

We now come to the following proposition, which, in particular, extends \eqref{product-rule} to $\Jscr\p{\Xscr_\Lm}_\rme$,  provided that  $\Lm\subseteq\bigcap_{p\geq2}L\p p(\mu)$.
\begin{proposition}\label{prop:con.sys}  Assume $X$ is a L\'evy process with characteristics $(\gm,\sig\p2,\nu)$.
Let $\Lm\subseteq\bigcap_{p\geq2}L\p p(\mu)$ and let  $\Xscr_\Lm$ denote  the associated family of martingales. 
Then 

\textnormal{(i)} $X\p \al$ possesses finite moments of every order, for every $\al\in\Lm$. 
 
\textnormal{(ii)} $\Lm$ is a total system in $L\p 2(\mu)$ if and only if $\Xscr_\Lm$ possesses the chaotic representation property with respect to  $\Fbb\p X$.

\textnormal{(iii)} The product formula \eqref{product-rule} and the moment formula \eqref{moment-formula} hold for the family $\Jscr\p{\Xscr_\Lm}_\rme$ of elementary iterated integrals generated by $\Xscr_\Lm$ .

\textnormal{(iv)} Let $F\p j\in\Escr\p{m_j}_T$, $j=1,\ldots,N$. Then the product formula \eqref{product-rule} and the moment formula \eqref{moment-formula} hold for the product $\prod_{j=1}\p N J_{m_j}\p{\al_{\overline m_{j-1}+1}: \al_{\overline m_j}}(F\p j)$.
\end{proposition}
\begin{proof}
First we verify (i). By \cite[Corollary 2.12]{R04}, there exists a constant $C_p>0$ such that
\begin{equation}\label{eq:Lp.est}
 \|X\p{\al}_T\|^p_{L\p{p}(\Pbb)}\leq C_pT \big(\|\al\|_{L\p2(\mu)}\p{p}+\|\al\|_{L\p{p}(\mu)}\p{p}\big).
\end{equation} 
Hence, $\Xscr_\Lm$ is a family of martingales with finite moments of every order if $\Lm\subseteq\bigcap_{p\geq2}L\p p(\mu)$. To see (ii) we refer to \cite[Theorem 6.6]{DTE16}.  
 To show (iii), let $\tilde\Lm:= L\p1(\mu)\cap\Bb_\Rbb$. Then  $\tilde \Lm$ satisfies Assumption \ref{ass:good.sys} and, according to Theorem \ref{thm:prod.rule}, 
the  multiplication formula \eqref{product-rule} holds for $\Jscr\p{\Xscr_{\tilde \Lm}}_\rme$. 
We now show by approximation that \eqref{product-rule} holds also for  $\Jscr\p{\Xscr_{\Lm}}_\rme$ if $\Lm\subseteq\bigcap_{p\geq2}L\p p(\mu)$.
We set $q_0:=2\p{\overline m_N}N$. By Lemma \ref{lem:bou.fun.app}, for every $\al\in\Lm$  there exists  a sequence $(\al\p k)_{k\geq1}\subseteq{\tilde\Lm}$ such 
that $\al\p k\longrightarrow \al$ in $L\p2(\mu)$ and in $L\p{q_0}(\mu)$ as $k\rightarrow+\infty$.  Since \eqref{poisson-integral} 
implies $X\p{\al\p k}_T-X\p{\al}_T=X\p{\al\p k-\al}_T$ a.s., by \eqref{eq:Lp.est} we get
\[
\|X\p{\al\p k}_T-X\p{\al}_T\|_{L\p{q_0}(\Pbb)}=\|X\p{\al\p k-\al}_T\|_{L\p{q_0}(\Pbb)}\longrightarrow0\quad \textnormal{ as }\ k\rightarrow+\infty.
\]
Let now  $(\al\p k_i)_{k\in\Nbb}\subseteq\tilde\Lm$ and $\al\p k_i\longrightarrow \al_i$ in $L\p2(\mu)$ and $L\p{q_0} (\mu)$, for every $i=1,\ldots,\overline m_N$. Then, setting $F^j:=F_{\otimes \overline m_{j-1}+1:\overline m_j}$ for every $j=1,\ldots,N$, Lemma \ref{lem:mom.it.int} (ii) with $p=N$ implies 
\begin{eqnarray*} 
\Ebb\Bigg[\sup_{t\in[0,T]}\Big|J\p{\al\p k_{\overline m_{j-1}+1:\overline m_j}}_{m_j}(F^j)_t-J\p{\al_{\overline m_{j-1}+1:\overline m_j}}_{m_j}(F^j)_t\Big|^{N}\Bigg]\longrightarrow0\quad\textnormal{ as } \ k\rightarrow+\infty.
\end{eqnarray*}
This yields, as a consequence of H\"older's inequality,
\begin{equation*}
\prod_{j=1}\p N J\p{\al\p k_{\overline m_{j-1}+1:\overline m_j}}_{m_j}(F^j)_t\longrightarrow \prod_{j=1}\p N J\p{\al_{\overline m_{j-1}+1:\overline m_j}}_{m_j}(F^j)_t\quad \textnormal{in }\ L\p1(\Pbb)\ \textnormal{ as }\ k\rightarrow+\infty,
\end{equation*}
for every $t\in[0,T]$.
For the  right-hand side of \eqref{product-rule}, we will show that 
$$\int_{M_t^{|{\bf s}|}} \bigg( \bigotimes_{j=1}^N   F\p j \bigg )_{\bf s}(t_1,\ldots,t_{|{\bf s}|}) 
  \rmd Y_{t_1}^{\al^{k,B}_{(S_1),i_1}} \ldots \rmd Y^{\al^{k,B}_{(S_{|{\bf s}|}), i_{|{\bf s}|}}}_{t_{|{\bf s}|}} $$
converges in $L\p2(\Pbb)$. For this, we first consider the convergence of  the processes $Y^{\al^{k,B}_{(S_r),i_r}}$. 

We define $Y^{\al^{B}_{(S_r),i_r}}$ according to \eqref{Y-in-detail}. We use the representation  \eqref{Y-in-detail} and discuss the convergence cases separately. 
Notice that, since we have chosen the sequence $\al\p k_j$ such that  $\al\p k_j \longrightarrow \al_j$ in $L^q(\mu)$   for every $2 \le q \le q_0=2\p{\overline m_N}N$, $j=1,\ldots,N$, and  $\mu = \sigma^2 \delta_0 + \nu$, 
it follows that
\equa
\nu\big(\al^{k}_{(S_r)}\big)  = \nu\Big(  \prod_{\ell \in S_r}  \al^{k}_{\ell}\Big)   \longrightarrow  \nu\Big(  \prod_{\ell \in S_r}  \al_{\ell}\Big)  = \nu\big(\al_{(S_r)}\big)\quad \text{ and } \quad \al^{k}_{(S_r)}(0) \longrightarrow  \al_{(S_r)}(0)\quad\ \text{ as }\ k\to +\infty.
\tion
Therefore, we also have   $ \al^k_{(S_r)}(0)   W^\sigma \longrightarrow  \al_{(S_r)}(0)   W^\sigma$ 
in $\Hscr\p{\,2}$ as $k\to+\infty$. Finally,  
$$X^{\widehat \al^k_{(S_r)}} =\Big(1_{ [0, \cdot]\times (\Rbb \setminus \{ 0\}) } \prod_{\ell \in S_r}  \al^{k}_{\ell}\Big) \ast  \tilde N \longrightarrow   \Big(1_{ [0, \cdot]\times (\Rbb \setminus \{ 0\}) } \prod_{\ell \in S_r}  \al_{\ell}\Big) \ast  \tilde N =  X^{\widehat \al_{(S_r)}}\ \text{ in }\ \Hscr\p2\ \text{ as }\ k\to+\infty, $$
since our assumptions imply that  $\widehat \al^k_{(S_r)} \longrightarrow \widehat \al_{(S_r)}$ in $L^2(\nu)$ as $k\to+\infty$.  Using a telescopic sum we estimate
\equal \label{telescop}
&&\less \Ebb\bigg[\bigg(\int_{M_t^{|{\bf s}|}} \bigg( \bigotimes_{j=1}^N   F\p j \bigg )_{\bf s}(t_1,\ldots,t_{|{\bf s}|}) 
  \rmd Y_{t_1}^{\al^{k,B}_{(S_1),i_1}} \ldots \rmd Y^{\al^{k,B}_{(S_{|{\bf s}|}) ,i_{|{\bf s}|}}}_{t_{|{\bf s}|}} \notag  \\
  && \quad \quad \quad -   \int_{M_t^{|{\bf s}|}}\bigg( \bigotimes_{j=1}^N   F\p j \bigg )_{\bf s}(t_1,\ldots,t_{|{\bf s}|}) 
  \rmd Y_{t_1}^{\al^{B}_{(S_1),i_1}} \ldots \rmd Y^{\al^{B}_{(S_{|{\bf s}|}), i_{|{\bf s}|}}}_{t_{|{\bf s}|}}      \bigg)\p2\bigg] \notag 
\\
&\leq&   C \sum_{r = 1}^N \Ebb\bigg[\bigg(\int_{M_t^{|{\bf s}|}} \bigg( \bigotimes_{j=1}^N   F\p j \bigg )_{\bf s}(t_1,\ldots,t_{|{\bf s}|}) 
  \rmd Y_{t_1}^{\al^{k,B}_{(S_1),i_1}} \ldots    \rmd Y_{t_1}^{\al^{k,B}_{(S_{r-1}),i_{r-1}}} \rmd (Y_{t_r}^{\al^{k,B}_{(S_1),i_r}}    -  Y_{t_r}^{\al^{B}_{(S_1),i_r}} ) \notag   \\
  &&\quad \quad \quad\quad \quad \quad \times  \rmd Y_{t_1}^{\al^{B}_{(S_{r+1}),i_{r+1}}} \ldots     \rmd Y^{\al^{B}_{(S_{|{\bf s}|}), i_{|{\bf s}|}}}_{t_{|{\bf s}|}}  \bigg)\p2\bigg] .
\tionl
 For the addends on the right-hand side of \eqref{telescop}, using It\^o's isometry whenever the integrator 
is an $\Hscr\p{\,2}$-martingale  and Cauchy-Schwarz inequality for integrators absolutely continuous with respect to $\rmd t$,  we find two constants 
\[
c\p k:=c(T,\al^{k,B}_{(S_1)i_1},\ldots,\al^{k,B}_{(S_{r-1})i_{r-1}},\al^{B}_{(S_{r+1})i_{r+1}},\ldots,\al^{B}_{(S_{|{\bf s}|}) i_{|{\bf s}|}},\sig,\nu)>0,\quad  C(T,\al^{k,B}_{(S_{r})i_{r}},\al^{B}_{(S_{r})i_{r}})>0,
\]  such that
\begin{equation}\label{eq:est.cau.seq.it.int}
\begin{split}
\Ebb\bigg[\bigg(\int_{M_t^{|{\bf s}|}} \bigg( \bigotimes_{j=1}^N   &F\p j \bigg )_{\bf s}(t_1,\ldots,t_{|{\bf s}|}) 
  \rmd Y_{t_1}^{\al^{k,B}_{(S_1)i_1}} \ldots    \rmd Y_{t_1}^{\al^{k,B}_{(S_{r-1})i_{r-1}}}\times\\&\times \rmd (Y_{t_r}^{\al^{k,B}_{(S_1)i_r}}    -  Y_{t_r}^{\al^{B}_{(S_1)i_r}} )\rmd Y_{t_1}^{\al^{B}_{(S_{r+1})i_{r+1}}} \ldots     \rmd Y^{\al^{B}_{(S_{|{\bf s}|}) i_{|{\bf s}|}}}_{t_{|{\bf s}|}}  \bigg)\p2\bigg] 
\\&\hspace{1cm}\leq
c\p k\,C(T,\al^{k,B}_{(S_{r})i_{r}},\al^{B}_{(S_{r})i_{r}})\,\int_{M_t^{|{\bf s}|}} \bigg ( \bigotimes_{j=1}^N   F\p j \bigg )_{\bf s}\p2(t_1,\ldots,t_{|{\bf s}|})\rmd  \lambda^{|{\bf s}|} (t_1,\ldots t_{|{\bf s}|}).
\end{split}
\end{equation}
We notice that the constant $c\p k$ remains bounded as $k\to+\infty$. Indeed, $c\p k$ can be explicitly computed and consists of products of terms of the form $\nu\big((\al^{k,B}_{(S_{j})i_{j}})\p2\big)$ or $\al^{k}_{(S_j)}(0)$, or $\nu\big(\al^{k,B}_{(S_{j})i_{j}}\big)$, as well as $T$ and  $\sigma^2$ for   $j=1,\ldots,r-1$, and similarly, but not dependent on $k$,  for $j=r+1,\ldots,|\bf s|$.  
 The constant $C\p k:=C(T,\al^{k,B}_{(S_{r})i_{r}},\al^{B}_{(S_{r})i_{r}})$ which arises from  It\^o's isometry or Cauchy-Schwarz inequality concerning the integrator $( Y_{t_r}^{\al^{k,B}_{(S_1)i_r}}    -  Y_{t_r}^{\al^{B}_{(S_1)i_r}} )$, can assume 
the values $ C\p k= \sigma^2 (\al^{k}_{(S_r)}(0)  -  \al^{k}_{(S_r)}(0) )^2 $ or   $ C\p k =\nu((\al^{k}_{(S_r)}  -\al_{(S_r)})^2 )$
when the integrator is in $\Hscr\p2, $ while we would have $C\p k= T (\nu(\al^{k}_{(S_r)} )  -   \nu(\al_{(S_r)} ))^2 $  or
$C\p k=T\sigma^4(\al^{k}_{(S_r)}(0)   -  \al_{(S_r)}(0))^2$ for deterministic integrators. But in any case, 
$$ C\p k=C(T,\al^{k,B}_{(S_{r})i_{r}},\al^{B}_{(S_{r})i_{r}}) \longrightarrow0  \quad \text{as} \quad  k\to+\infty. $$

The proof of the product formula for this more general case is complete. For the moment formula \eqref{moment-formula} we observe that this is a direct consequence of the product formula \eqref{product-rule} and the proof can be given as in Corollary \ref{cor:moment.formula}. Clearly (iv) is a direct consequence of (iii) because of the linearity of the iterated integrals and the multi-linearity of the tensor product. The proof of the proposition is complete.

\end{proof}
We now generalize \eqref{product-rule}  and  \eqref{moment-formula} to the case in which the functions  $F^j$ need not to be bounded but, rather, satisfy some   integrability condition. This is the main result of the present paper. For the definition of ${\bf s}$ recall Definition \ref{partitions} and the relations \eqref{s-no-zero} and \eqref{s-rule}.

 \begin{theorem}\label{thm:gen.it.in}  
 Let $X$ be a L\'evy process with  characteristic triplet $(\gm,\sig\p2,\nu)$. 
Let   $\al_1,\ldots, \al_{\, \overline  m_N}$ belong to $ \bigcap_{p\geq2}L\p p(\mu)$
and  assume that $F\p j  \in L\p2(M_T\p{m_j},\rho\p{\al_{\, \overline m_{j-1}+1:\overline m_j}})$, $j=1,
\ldots,N$, are such that
\begin{equation}\label{eq:gen.it.in.cond}
\int_{M_T^{|{\bf s}|}} \bigg ( \bigotimes_{j=1}^{ N} F^j \bigg )^2_{\bf s}(t_1,\ldots,t_{|{\bf s}|}) 
  \rmd\lm( t_1,\ldots, t_{|{\bf s}|}) < \infty, \quad \forall\, B \subseteq [\overline m_N], \,  \forall\, {\bf s} \in  \Pi_{\le 2, \ge 1}^{}(B,B^c; m_1,\ldots,m_N).
\end{equation}
Then, the product formula \eqref{product-rule} and the moment formula \eqref{moment-formula}  extend to $\prod_{j=1}\p NJ_{m_j}\p{\al_{\overline m_{j-1}+1: \overline m_j}}(F\p j)_t$ for every $t\in[0,T].$
\end{theorem}
Condition \eqref{eq:gen.it.in.cond} is an $L\p2$-bound on $(\bigotimes_{j=1}^{ N} F^j)_{\bf s}$ and ensures that all the stochastic integrals appearing on the right-hand side of \eqref{product-rule} are square integrable martingales.

 Notice that \eqref{eq:gen.it.in.cond} is fulfilled, for example,  whenever $F\p j\in L\p {2N}(M_T\p{m_j},\rho\p{\al_{\, \overline m_{j-1}+1:\overline m_j}})$ for any $j=1,\ldots,N$, (see Remark \ref{sufficient} below).

 \begin{proof}[Proof of Theorem \ref{thm:gen.it.in}]
We choose $\Lm \subseteq\bigcap_{p\geq2}L\p p(\mu)$ such that $\al_1,\ldots, \al_{\, \overline  m_N} \in \Lm$ and recall that, because of Proposition \ref{prop:con.sys} (i), the elements of $\Xscr_\Lm$ possess moments of every order.  From Proposition \ref{prop:con.sys} (iv), we know that the product rule \eqref{product-rule} extends to the case $\Lm\subseteq\bigcap_{p\geq2}L\p p(\mu)$ for $F\p j\in \Escr_T^{m_j}$,  $j=1,\ldots,N$. For later use, we derive an estimate for the  second moment of the iterated integrals on the right hand side of this extension of \eqref{product-rule}. Let  $F\p j \in \Escr_T^{m_j}$, $j=1,\ldots,N$, and $\Lm\subseteq\bigcap_{p\geq2}L\p p(\mu)$. Applying It\^o's isometry every time the process $Y^{\al^{B}_{(S_{j}) i_{j}}}$ is random, and  H\"older's inequality whenever it is deterministic, we find a constant $c(T,\al^{B}_{(S_1)i_1},\ldots,\al^{B}_{(S_{|{\bf s}|}) i_{|{\bf s}|}},\sig,\nu)>0$,  such that
\begin{equation}\label{eq:est.cau.seq.it.int.fun}
\begin{split}
\Ebb\bigg[\bigg(\int_{M_t^{|{\bf s}|}}& \bigg( \bigotimes_{j=1}^N   F\p j \bigg )_{\bf s}(t_1,\ldots,t_{|{\bf s}|}) 
  \rmd Y_{t_1}^{\al^{B}_{(S_1),i_1}} \ldots \rmd Y^{\al^{B}_{(S_{|{\bf s}|}), i_{|{\bf s}|}}}_{t_{|{\bf s}|}}\bigg)\p2\bigg]
\\&\leq
c(T,\al^{B}_{(S_1)i_1},\ldots,\al^{B}_{(S_{|{\bf s}|}) i_{|{\bf s}|}},\sig,\nu)\int_{M_t^{|{\bf s}|}} \bigg ( \bigotimes_{j=1}^N   F\p j \bigg )_{\bf s}\p2(t_1,\ldots,t_{|{\bf s}|})\rmd  \lambda^{|{\bf s}|} (t_1,\ldots t_{|{\bf s}|}).
\end{split}
\end{equation}  
The estimate \eqref{eq:est.cau.seq.it.int.fun} implies that we may extend by linearity and continuity the iterated integral appearing on the left-hand side of \eqref{eq:est.cau.seq.it.int.fun} to those functions $F\p j \in L\p2(M_T\p{m_j},\rho\p{\al_{\overline m_{j-1}+1:\overline m_j}}), \, j=1,\ldots,N$, which furthermore satisfy the integrability condition \eqref{eq:gen.it.in.cond}. We now divide the proof into three steps.

\noindent \underline{ Step 1.} In this first step we are going to show that the product formula \eqref{product-rule} extends to indicator functions $F\p j =1_{A_j} \in L\p2(M_T\p{m_j},\rho\p{\al_{\overline m_{j-1}+1:\overline m_j}})$. 
  By  \cite[Theorem 2.40]{Foll}, given  an $\varepsilon>0$ and a Borel set $A_j \in \mathcal{B}([0,T]^{m_j})$,  there exists an $M \ge 1$ and  disjoint  rectangles    
$R_{j,1},\ldots R_{j,M} \subseteq M_T\p{m_j}$ whose sides are intervals such that 
$$ \int_{M_T^{m_j}} \Big| 1_{A_j} -\sum_{i=1}^M 1_{R_{j,i}} \Big | \rmd\lm_{m_j} = \lm_{m_j}\Big(A_j \Delta \bigcup_{i=1}^M R_{j,i}\Big)   < \varepsilon,$$ 
where $A\Delta B:=(A\setminus B)\cup(B\setminus A)$ denotes the symmetric difference of $A$ and $B,$ and we used the identity $1_{A\Delta B}=|1_A-1_B|$.
 We notice that $\sum_{i=1}^M 1_{R_{j,i}} \in \Escr_T^{m_j}$  and $0\le \sum_{i=1}^M 1_{R_{j,i}} \le 1$.  Hence, since  $\rho\p{\al_{\overline m_{j-1}+1:\overline m_j}}$ is absolutely continuous with respect to $\lm\p{\, m_j}$, there are  sequences $(F\p j_k)_{k=1}^\infty  \subseteq \Escr_T^{m_j}$    with  $|F\p j_k| \le 1$ such that 
\equal \label{step-indicator}
 F\p j_k \longrightarrow  1_{A_j} \quad  \text{ in }  \quad L\p 2(M_T\p{m_j},\rho\p{\al_{\overline m_{j-1}+1:\overline m_j}}),\ \text{ as }\ k\to+\infty,\quad j=1,\ldots,N.
 \tionl
To estimate the difference below we first use a telescopic sum, apply the Cauchy-Schwartz inequality and then use  the product formula \eqref{product-rule} from Proposition \ref{prop:con.sys} (iv) so that 
\equal \label{product-formula-used}
&&\EE\bigg [ \bigg ( \prod_{j=1}^N J_{m_j}^{\al_{\overline m_{j-1}+1: \overline m_j}}(F^{j}_k)_t   -  \prod_{j=1}^N J_{m_j}^{\al_{\overline m_{j-1}+1: \overline m_j}}(F^{j}_\ell)_t\bigg )^2 \bigg ] \notag \\
&\le & c \sum_{r=1}^N  \EE \bigg [ \bigg ( \prod_{j=1}^r J_{m_j}^{\al_{\overline m_{j-1}+1: \overline m_j}}(F^{j}_k)_t  \,\, J_{m_{r+1}}^{\al^k_{\overline m_{r}+1: \overline m_{r+1}}}(F^{r+1}_k - F^{r+1}_\ell)_t   \,\, \prod_{i=r+2}^N J_{m_i}^{\al_{\overline m_{i-1}+1: \overline m_i}}(F^{i}_\ell)_t\bigg )^2 \bigg ]  \notag \\
&=& \!\!\!\! c \sum_{r=1}^N \EE \bigg [\bigg (  \sum_{B \subseteq [\overline m_N]} \,\,  \sum_{{\bf s} \in  \Pi_{\le 2, \ge 1}^{}(B,B^c; m_1,\ldots,m_N)} \,\,\  \sum_{ {\bf i}  \in  I_{{\bf s}}}  
 \int_{M_t^{|{\bf s}|}} \bigg( \bigotimes_{j=1}^r   F\p j_k \otimes     (F^{r+1}_k - F_\ell^{r+1})   \otimes \bigotimes_{i=r+2}^N    F\p i_\ell\bigg )_{\bf s}(t_1,\ldots,t_{|{\bf s}|}) \notag  \\
 && \hspace{8em} \times \rmd Y_{t_1}^{\al^{B}_{(S_1)i_1}} \ldots \rmd Y^{\al^{B}_{(S_{|{\bf s}|}) i_{|{\bf s}|}}}_{t_{|{\bf s}|}} \bigg )^2\bigg ],
\tionl 
where $c>0$ is a constant. By \eqref{eq:est.cau.seq.it.int.fun} we conclude that the right-hand side of \eqref{product-formula-used} converges to zero. Indeed,  since the integrand is bounded, by dominated convergence  and \eqref{step-indicator} we get
\equal \label{tensor-Holder}
&& \int_{M_t^{|{\bf s}|}} \bigg ( \bigotimes_{j=1}^r   F\p j_k \otimes     (F^{r+1}_k - F^{r+1}_\ell)   \otimes \bigotimes_{i=r+2}^N  F\p i_\ell\bigg )_{\bf s}\p2\rmd  \lambda^{|{\bf s}|} (t_1,\ldots t_{|{\bf s}|})  \longrightarrow 0\ \text{ as }\ k\to+\infty.
\tionl 

Hence, we have shown that  ${  Z}\p k:= \prod_{j=1}^N J_{m_j}^{\al_{\overline m_{j-1}+1: \overline m_j}}(F^{j}_k)_t$ is a Cauchy sequence in  $L\p2(\Pbb)$.  So, there exists $Z\in L\p2(\Pbb)$ such that $Z\p k\longrightarrow Z$ in  $L\p2(\Pbb)$ as $k\to\infty$. On the other side, since  $J_{m_j}^{\al_{\overline m_{j-1}+1: \overline m_j}}(F^{j}_k)_t$ converges to $ J_{m_j}^{\al_{\overline m_{j-1}+1: \overline m_j}}(F^{j})_t$ in $L\p2(\Pbb)$  as $k\to\infty$, we also have that $Z\p k\longrightarrow\prod_{j=1}^N J_{m_j}^{\al_{\overline m_{j-1}+1: \overline m_j}}(F^{j})_t$ \emph{in probability} as $k\to+\infty$. This implies 
\begin{equation}\label{eq:id.lim.prod}
Z=\prod_{j=1}^N J_{m_j}^{\al_{\overline m_{j-1}+1: \overline m_j}}(F^{j})_t\ \quad  \text{ a.s.},
\end{equation} 
because of the uniqueness of the limit in probability. Thus, $\prod_{j=1}^N J_{m_j}^{\al_{\overline m_{j-1}+1: \overline m_j}}(F^{j})_t\in L\p2(\Pbb)$ for every $t\in[0,T]$. From Proposition \ref{prop:con.sys} (iv), we know that $Z\p k$ satisfies the product formula \eqref{product-rule} for every $k$. We now discuss the convergence of the corresponding right hand side in \eqref{product-rule}  for $Z\p k$.   Similarly as for \eqref{product-formula-used} but now with $F\p j$  instead of and $F\p j_\ell,$  we see that
\equal \label{integral-difference}
\EE \bigg [ \bigg (\int_{M_t^{|{\bf s}|}} \bigg( \bigotimes_{j=1}^N   F\p j_k  -  \bigotimes_{j=1}^N    F\p j\bigg )_{\bf s}(t_1,\ldots,t_{|{\bf s}|})  
  \rmd Y_{t_1}^{\al^{B}_{(S_1)i_1}} \ldots \rmd Y^{\al^{B}_{(S_{|{\bf s}|}) i_{|{\bf s}|}}}_{t_{|{\bf s}|}} \bigg )^2 \bigg ] \longrightarrow 0\quad\text{as }\ k\to+\infty.
\tionl
 This means that also $Z$ satisfies  \eqref{product-rule}. Thus, \eqref{product-rule} holds for $ F^j= 1_{A_j} \in  L\p2(M_T\p{m_j},\rho\p{\al_{\overline m_{j-1}+1:\overline m_j}}).$ The proof of the first step is complete. 
 
 \noindent \underline{ Step 2.}  We observe that, as a consequence of   Step 1,  by linearity, the product formula \eqref{product-rule} also holds for simple functions from $L\p2(M_T\p{m_j},\rho\p{\al_{\overline m_{j-1}+1:\overline m_j}})$.

\noindent \underline{ Step 3.}
 In  this  step we show the product formula  \eqref{product-rule} for arbitrary functions $F\p j\in L\p2(M_T\p{m_j},\rho\p{\al_{\overline m_{j-1}+1:\overline m_j}})$, $j=1,\ldots, N$, which furthermore satisfy the integrability condition \eqref{eq:gen.it.in.cond}.
 Any  $F\p j\in L\p2(M_T\p{m_j},\rho\p{\al_{\overline m_{j-1}+1:\overline m_j}})$ can be pointwise approximated by a  sequence $(F\p j_k)_{k\geq1} $ of simple functions  such that $|F\p j_k| \uparrow |F\p j|$ as $k\to+\infty$ (see, for example, \cite[Theorem 2.10b]{Foll}). Because of this and \eqref{eq:gen.it.in.cond}, we get
\begin{equation}\label{eq:approx.id.rule}
\int_{M_T^{|{\bf s}|}} \bigg ( \bigotimes_{j=1}^{ N} F^j_k \bigg )^2_{\bf s}(t_1,\ldots,t_{|{\bf s}|}) 
  \rmd\lm( t_1,\ldots, t_{|{\bf s}|})\leq\int_{M_T^{|{\bf s}|}} \bigg ( \bigotimes_{j=1}^{ N} F^j \bigg )^2_{\bf s}(t_1,\ldots,t_{|{\bf s}|}) 
  \rmd\lm( t_1,\ldots, t_{|{\bf s}|})<+\infty.
\end{equation}
 This implies that we may extend by linearity and continuity the iterated integrals on the right hand side of \eqref{product-rule}  defined for simple functions $(F\p j_k)_{k\geq1} $ to  those $F\p j \in L\p2(M_T\p{m_j},\rho\p{\al_{\overline m_{j-1}+1:\overline m_j}}), \, j=1,\ldots,N$, which satisfy  \eqref{eq:gen.it.in.cond}.  Using the second step,  we  can conclude like in the first step that  \eqref{product-formula-used}  holds also for simple functions.   The difference is that now the integrand   \[ \bigg( \bigotimes_{j=1}^r   F\p j_k \otimes     (F^{r+1}_k - F_\ell^{r+1})   \otimes \bigotimes_{i=r+2}^N    F\p i_\ell\bigg )_{\bf s}\]
is not bounded. However, because of \eqref{eq:approx.id.rule}, we can apply dominated convergence to get \eqref{tensor-Holder}. Clearly, we can identify the $L\p2(\Pbb)$-limit  $Z$ of the corresponding Cauchy sequence  $Z\p k$ as in \eqref{eq:id.lim.prod}. Because of the  second step,  $Z\p k$ satisfies  \eqref{product-rule}.  To complete the proof 
we verify that also for the general case we have the correct limit expression on the right hand side of \eqref{product-rule} and, hence, that the multiplication formula holds also for the limit  $Z$.  An estimate like   \eqref{eq:est.cau.seq.it.int.fun} ensures that  the addends on the right hand side of \eqref{product-rule}  belong all to  $L\p2(\Pbb)$ for any
$F\p j\in L\p2(M_T\p{m_j},\rho\p{\al_{\overline m_{j-1}+1:\overline m_j}})$, $j=1,\dots,N$, such that \eqref{eq:gen.it.in.cond} holds. Hence, \eqref{integral-difference} follows from  \eqref{eq:gen.it.in.cond} and \eqref{eq:approx.id.rule}, because of  $|F\p j_k| \uparrow |F\p j|$ as $k\to\infty$, $j=1,\ldots,N$. This  proves  \eqref{product-rule} for the general case. 

To see \eqref{moment-formula}, because of \eqref{product-rule} and since each summand on the  right-hand side of \eqref{product-rule} belongs to $L\p2(\Pbb)$,  we can proceed as in the proof of Corollary \ref{cor:moment.formula}.  The proof of the theorem is now complete.
 \end{proof}

\begin{remark} \label{sufficient} A sufficient condition for   \eqref{eq:gen.it.in.cond}  
is  
$$ F^j \in  L\p{2N}(M_T\p{m_j}, \lm^{m_j}), \quad j=1,\ldots,N.  $$
 To see this, we recall  that, from \eqref{s-rule}, to apply the identification rule ${\bf s}: H(u_1,\ldots,u_{\overline m_N}) \mapsto H_{{\bf s}}(t_1,\ldots,t_k)$, we need $H$ to depend on $u_1,\ldots,u_{\overline m_N}.$  Therefore, we extend each $F^j$ which depends 
on the variables $u_{\overline m_{j-1}+1},\ldots, u_{\overline m_{j}}$  constantly  to all  variables $u_1,\ldots,u_{\overline m_N}$ and denote the extension by $\widehat 
F^j(u_1,\ldots,u_{\overline m_N}).$ Since  
$$|S_\ell \cap \{u_{\overline m_{j-1}+1},\ldots, u_{\overline m_{j}}\}| \le 1, $$   
the identification rule 
$$   {\bf s} : \widehat F^j(u_1,\ldots,u_{\overline m_N}) \mapsto \widehat F_{{\bf s}}^j(t_1,\ldots,t_{|{\bf s}|})$$
causes  only a one-one {\it renaming} but no identification among the variables $u_{\overline m_{j-1}+1},\ldots, u_{\overline m_{j}}.$ We have
$$ \bigg ( \bigotimes_{j=1}^{ N} F^j \bigg )_{\bf s}(t_1,\ldots,t_{|{\bf s}|}) =  (\widehat F^1)_{{\bf s}}(t_1,\ldots,t_{|{\bf s}|})  
\times \ldots \times   (\widehat F^N)_{{\bf s}} (t_1,\ldots,t_{|{\bf s}|}).$$
Consequently, by H\"older's  inequality,  
\equa
 \int_{M_T^{|{\bf s}|}} \bigg ( \bigotimes_{j=1}^{ N} F^j \bigg )^2_{\bf s}(t_1,\ldots,t_{|{\bf s}|}) 
  \rmd\lm( t_1,\ldots, t_{|{\bf s}|})  
&\le &  \prod_{j=1}^N   \left (  \int_{M_T^{|{\bf s}|}}  (  \widehat F^j )^{2N}_{\bf s} (t_1,\ldots,t_{|{\bf s}|}) 
  \rmd\lm( t_1,\ldots, t_{|{\bf s}|}) \right )^\frac{1}{N} \\
&\le &   \prod_{j=1}^N T^{|{\bf s}|-m_{j}} \left (\int_{M_T^{m_j}}  ( F^j  )^{2N}(t_1,\ldots, t_{m_{j}}) 
  \rmd\lm(t_1,\ldots, t_{m_{j}})  \right )^\frac{1}{N}.    \\
\tion
\end{remark}

\subsection{Examples for  $\Xscr_\Lm$}
We conclude this section giving  examples of  $\Xscr_\Lm$ satisfying the assumptions of Theorem \ref{thm:gen.it.in}. We observe that, if $\al\in L\p2(\mu)$, then the L\'evy measure $\nu\p \al$ of the square integrable martingale and $\Fbb$-L\'evy process $X\p \al$ is the image measure of $\nu$ under the mapping $\al$ (see \cite{Ba01}, Definition 7.6), that is,
\begin{equation}\label{eq:le.m.en.mar}
\nu\p \al(\rmd x)=(\nu\circ \al\p{-1})(\rmd x).
\end{equation}
The main point of this part is to construct families of martingales $\Xscr_\Lm$ possessing moments of every order. Because of the equivalence between the totality of a system $\Lm$ in $L\p2(\mu)$  and the CRP of the family $\Xscr_\Lm$, we will  consider systems $\Lm$ which are total in $L\p2(\mu)$ or, more specifically, orthogonal bases of $L\p2(\mu)$. 

\paragraph*{Dyadic Intervals. } This example is very simple but it is interesting because it holds without further assumptions on the L\'evy measure. Let $X$ be a L\'evy process relative to $\Fbb$ and let $(\gm,\sig\p2,\nu)$ denote its characteristic triplet. Let $\Dscr$ denote the set of dyadic numbers and define 

\[
\tilde\Lm:=\big\{1_{(a,b]}1_{\{0\}}+1_{(a,b]}1_{\Rbb\setminus\{0\}},\ a,b\in\Dscr:\ 0\notin[a,b]\big\}\cup\{0\}.
\]
Then $\tilde\Lm$ is total in $L\p2(\mu)$, consists of countably many functions and satisfies Assumption \ref{ass:good.sys}. Clearly, we can orthonormalize $\tilde\Lm$ in $L\p2(\mu)$ and obtain an orthonormal basis $\Lm$ of $L\p2(\mu)$ consisting of bounded functions. The associated family $\Xscr_\Lm$ consists of countably many orthogonal martingales and it possesses the CRP with respect to $\Fbb\p X$ (cf.\ Proposition \ref{prop:con.sys}) in the simpler version of Remark \ref{rem:alt.rep.crp}. Furthermore, from Theorem \ref{thm:gen.it.in}, we get \eqref{product-rule} and \eqref{moment-formula} for the iterated integrals generated by $\Xscr_\Lm$. 

\paragraph*{Teugels martingales. }  Teugels martingales were already discussed in \S\ref{subs:exa}. We notice that, to introduce Teugels martingales as square integrable martingales, it is enough that $X$ possesses moments of arbitrary order, that is, all monomials $p_n(x):=x\p n$ belong to $L\p2(\nu)$, for every $n\geq1$. We assume that there exist constants $\lm,\ep>0$ such that $x\mapsto1_{(-\ep,\ep)}  \rme\p{\lm |x|/2} $ belongs to $L\p2(\nu)$ and define $h_1(x):=1_{\{0\}}+1_{\Rbb\setminus\{0\}}(x)p_1(x)$ and $h_n(x):=1_{\Rbb\setminus\{0\}}(x)p_n(x)$, $n\geq2$. The system $\tilde\Lm:=\{h_n,\ n\geq1\}$ belongs to $L\p2(\mu)$ and it is total. Furthermore, 
the identity $X\p{(i)}=X\p{h_i}$ holds, where $X\p{(i)}$ denotes the $i$th Teugels martingale as introduced in \S\ref{subs:exa}. The associated family $\Xscr_{\tilde\Lm}$ of Teugels martingales is compensated-covariation stable and, according to Proposition \ref{prop:con.sys} (ii), it possesses the CRP with respect to $\Fbb\p X$. However, the system $\tilde\Lm$ does not satisfy Assumption \ref{ass:good.sys} because it does not consist of bounded functions  and it is not stable under multiplication. Let $\nu\p{h_n}$ be the L\'evy measure of the martingale $X\p{h_n}$. Then, from \eqref{eq:le.m.en.mar}, we get
\[
\int_\Rbb |x|\p m\nu\p{h_n}(\rmd x)=\int_\Rbb |x|\p {m+n}\nu(\rmd x)<+\infty.
\]
Hence $\tilde\Lm\subseteq\bigcap_{p\geq2}L\p p(\mu),$ and the family $\Xscr_{\tilde\Lm}$ is contained in $\bigcap_{p\geq2}\Hscr\p p$. From Theorem \ref{thm:gen.it.in} we get \eqref{product-rule} and \eqref{moment-formula} for the iterated integrals generated by $\Xscr_{\tilde\Lm}$. The family $\Xscr_{\tilde\Lm}$ is the family of Teugels martingales. We can orthonormalize ${\tilde\Lm}$ in $L\p2(\mu)$ to get an orthonormal system $\Lm$ consisting of polynomials. Therefore, the associated family of martingales $\Xscr_\Lm$ consists of countably many orthogonal martingales  and satisfies the  CRP with respect to $\Fbb\p X$ in the simpler form of Remark \ref{rem:alt.rep.crp}. Moreover,  the product and the moment formula hold  for the iterated integrals generated by $\Xscr_{\Lm}$.

\paragraph*{Hermite polynomials. }  Assume that $X$ is a L\'evy process with L\'evy measure $\nu$ which is equivalent to the Lebesgue measure on $\Rbb$, that is,
\begin{equation}\label{def:eq.levy.meas}
\nu(\rmd x)=h(x)\rmd x,\quad h(x)>0.
\end{equation}
This is, for example, the case if $X$ is an $\al$-stable L\'evy process (see \cite[Chapter 3]{S99}). In this case, if $(H_n)_{n\geq1}$ is the family of Hermite polynomials, setting $g(x):=(h(x))\p{-1/2}\rme\p{-x\p2/2}$, the family $\Lm:=\{P_n,\ n\geq1\}$, where $P_n:=1_{\{0\}}H_n(0)+1_{\Rbb\setminus\{0\}}gH_n$   is an orthogonal basis of $L\p2(\mu)$. However,  in general, we cannot expect that the family $\Xscr_{\Lm}$ has moments of every order, that is, that the inclusion $\Lm\subseteq\bigcap_{p\geq2}L\p p(\mu)$ holds. From \eqref{eq:le.m.en.mar}, using the definition of $g$ and of Hermite polynomials (see \cite[\S1.1]{N06}) we see that 
\[
\int_{\Rbb\setminus\{0\}}|x|\p{2m}\nu\p{P_n}(\rmd x)=\frac{1}{n!}\,\int_{\Rbb\setminus\{0\}} \big(h(x)\big)\p{1-m}\rme\p{-mx\p2}|q_{2m+n}(x)|\rmd x
\]
where $q_n$ denotes a polynomial function of order $n$, $n\geq1$. Therefore, to ensure that $\Xscr_\Lm$ has finite moments of every order, we derive the condition on $h$
\begin{equation}\label{eq:con.h.Hp}
\int_{\Rbb\setminus\{0\}} \big(h(x)\big)\p{1-m}\rme\p{-mx\p2}|x|^{2m+n}\rmd x<+\infty,\qquad \textnormal{\ for every\ } m,n\geq1.
\end{equation}
Hence, if $h$ satisfies \eqref{eq:con.h.Hp}, the moment formula \eqref{moment-formula} holds for the iterated integrals generated by the orthogonal family $\Xscr_\Lm$. 

Observe that $\al$-stable processes are an example of a family of L\'evy processes satisfying \eqref{eq:con.h.Hp}. We remark that, for an $\al$-stable L\'evy process $X$, the family $\Xscr_\Lm$ has finite moments of every order but the L\'evy process $X$ itself may possess infinite moments. This is the case, for example, if $X$ is a Cauchy process.

\paragraph*{Haar basis. } 
We consider a L\'evy process $X$ with L\'evy measure $\nu$ as in \eqref{def:eq.levy.meas}. 
Let $\lm$ be the Lebesgue measure on $(\Rbb,\Bscr(\Rbb))$ and  $\{\psi_{jk}:=2\p {\frac j2}\,\psi(2\p j\,x-k),\ x\in\Rbb,\ j,k\in\Zbb\}\subseteq L\p2(\lm)$ be the Haar basis. The generic element of the Haar basis can be written as 
\[
\textstyle\psi_{jk}(x):=2\p{\frac j2}\Bigl[1_{\left[\frac{k}{2\p{j}},\frac{2k+1}{2\p{j+1}}\right)}(x)-1_{\left[\frac{2k+1}{2\p{j+1}},\frac{k+1}{2\p{j}}\right)}(x)\Bigr]
\]
and its support is $\{\psi_{jk}\neq0\}=[\frac{k}{2\p{j}},\frac{k+1}{2\p{j}})$, $j,k\in\Zbb$.   
Note that the boundary points of these intervals are dyadic rational numbers. Defining $\Lm:=\{\psi_{jk}(0)1_{\{0\}}+1_{\Rbb\setminus\{0\}}h\p{-1/2}\psi_{jk},\ j,k\in\Zbb\}$   we obtain an orthogonal basis of $L\p2(\mu)$. The system $\Xscr_\Lm$ is a family of orthogonal martingales.  In general, the martingales in $\Xscr_\Lm$ do not have finite moments of every order. Let us denote be $\nu\p{\psi_{jk}}$ the L\'evy measure of the martingales $X\p{\psi_{jk}}$. Then, from \eqref{eq:le.m.en.mar}, we get
\[
\int_{\Rbb\setminus\{0\}} |x|\p{2m}\nu\p{\psi_{jk}}(\rmd x)=\int_{\Rbb\setminus\{0\}} (h(x))\p{1-m}|\psi_{jk}(x)|\p{2m}\nu(\rmd x).
\]
That is, we obtain the following condition on $h$
\begin{equation}\label{eq:con.h.Hw}
\int_a\p b (h(x))\p{1-m}\nu(\rmd x)<+\infty,\qquad \textnormal{for all dyadic rational numbers }\ a\ \textnormal{and }\ b:\quad 0\notin[a,b].
\end{equation}
Hence, if $h$ satisfies \eqref{eq:con.h.Hw}, we get \eqref{moment-formula} for the iterated integrals generated by $\Xscr_\Lm$. Notice that \eqref{eq:con.h.Hw} is satisfied, for example, if $h>0$ is bounded over all intervals whose extremes are dyadic rational numbers. We again find that if $X$ is an $\al$-stable L\'evy process,  condition \eqref{eq:con.h.Hw} is satisfied.

\paragraph{Acknowledgement.} PDT acknowledges Martin Keller-Ressel (TU Dresden) and funding from the German Research Foundation (DFG) under grant ZUK 64.

\end{document}